\documentclass[11pt,a4paper,reqno]{amsart}

\usepackage{amsfonts}
\usepackage{amsmath}
\usepackage{amssymb}
\usepackage{amsthm}
\usepackage{amsxtra}
\usepackage{bbm}
\usepackage{braket}
\usepackage{comment}
\usepackage{extarrows}
\usepackage{latexsym}

\usepackage{mathrsfs}
\usepackage{mathtools}
\usepackage{verbatim}

\usepackage[pdfborder={0 0 0}]{hyperref} 
\hypersetup{breaklinks, raiselinks}

\usepackage{bm}

\usepackage[initials,lite]{amsrefs} 
\AtBeginDocument{%
    \def\MR#1{}
}

\usepackage{cleveref}

\usepackage{a4wide}
\usepackage{graphicx,supertabular,paralist}
\usepackage[usenames,dvipsnames]{xcolor}

\usepackage{tikz}
\usetikzlibrary{matrix}
\usetikzlibrary{arrows, patterns, bending}
\tikzstyle{ghostfill} = [fill=white]
         \tikzstyle{ghostdraw} = [draw=black!50]
\usetikzlibrary{arrows,shapes,positioning}
\usetikzlibrary{decorations.markings}
\tikzstyle arrowstyle=[scale=1]
\tikzstyle directed=[postaction={decorate,decoration={markings, mark=at position .65 with {\arrow[arrowstyle]{stealth}}}}]
\tikzstyle reverse directed=[postaction={decorate,decoration={markings, mark=at position .65 with {\arrowreversed[arrowstyle]{stealth};}}}]

\def\Z{{\mathbb Z}}

\def\p{{\mathfrak{p}}}

\def\lcm{\operatorname{lcm}}

\newcommand{\reg}{\operatorname{reg}_R}

\def\p{{\mathfrak p}}

\theoremstyle{plain}
\newtheorem{Theorem}{Theorem}[section]
\newtheorem{Lemma}[Theorem]{Lemma}
\newtheorem{Corollary}[Theorem]{Corollary}
\newtheorem{Proposition}[Theorem]{Proposition}

\theoremstyle{definition}

\newtheorem{Assumptions and Discussion}[Theorem]{Assumptions and Discussion}

\newtheorem{Example}[Theorem]{Example}
\newtheorem{Definition}[Theorem]{Definition}

\theoremstyle{remark}
\newtheorem{observation}[Theorem]{Observation}
\newtheorem{remark}[Theorem]{Remark}
\newtheorem{Remark}[Theorem]{Remark}

\newtheorem*{acknowledgement*}{Acknowledgement}

\newtheorem*{ex*}{Example}
\newtheorem*{exer*}{Exercise}
\newtheorem*{rem*}{Remark}
\newtheorem*{prob*}{Problem}
\newtheorem*{prop*}{Proposition}
\newcommand{\pd}{\mathop{\mathrm{pd}}\nolimits}

\usepackage[shortlabels]{enumitem}
\SetEnumerateShortLabel{a}{\textup{(\alph*)}}
\SetEnumerateShortLabel{A}{\textup{(\Alph*)}}
\SetEnumerateShortLabel{1}{\textup{(\arabic*)}}
\SetEnumerateShortLabel{i}{\textup{(\roman*)}}
\SetEnumerateShortLabel{I}{\textup{(\Roman*)}}

\def\ker{\operatorname{ker}}
\def\lcm{\operatorname{lcm}}

\def\Sta{\operatorname{Sta}} 
\def\StaG{\operatorname{StaG}} 

\def\supp{\operatorname{supp}}

\def\KK{{\mathbb K}}

\def\NN{{\mathbb N}}

\def\QQ{{\mathbb Q}}

\def\ZZ{{\mathbb Z}}

\def\calF{\mathcal{F}}

\usepackage{bm}

\def\bda{{\bm a}}

\def\bde{{\bm e}}

\def\bdK{{\bm K}}

\def\bdN{{\bm N}}

\def\bdx{{\bm x}}

\def\bdalpha{{\bm \alpha}}
\def\bdbeta{{\bm \beta}}

\def\deg{\operatorname{deg}}

\def\upT{\textup{T}}

\usepackage{floatrow}
\usepackage{caption}
\DeclareCaptionSubType[alph]{figure}
\usepackage{tikz}
\usetikzlibrary{calc,arrows,shapes,positioning,intersections,through,backgrounds}
\usepackage{tkz-arith}
\usepackage{tkz-graph}
\usepackage{tkz-berge}
\usepackage{tikz-cd} 
\usetikzlibrary{matrix}
\usetikzlibrary{arrows, patterns}
\tikzstyle{ghostfill} = [fill=white]
\tikzstyle{ghostdraw} = [draw=black!50]
\usetikzlibrary{arrows,shapes,positioning}
\usetikzlibrary{decorations.markings}
\tikzstyle arrowstyle=[scale=1]
\tikzstyle directed=[postaction={decorate,decoration={markings, mark=at position .5 with {\arrow[arrowstyle]{stealth}}}}]
\tikzstyle reverse directed=[postaction={decorate,decoration={markings, mark=at position .5 with {\arrowreversed[arrowstyle]{stealth};}}}]

\begin{document}
\baselineskip=16pt

\title[Gen Newton Dual]
{\Large\bf Generalized Newton Complementary Duals of Monomial Ideals}

\author[Katie Ansaldi, Kuei-Nuan Lin and Yi-Huang Shen]
{Katie Ansaldi,  Kuei-Nuan Lin and Yi-Huang Shen}

\thanks{AMS 2010 {\em Mathematics Subject Classification}.
Primary 13D02; Secondary 05E40.}


\address{
 {Department of Mathematics and Computer Science, Wabash college, Crawfordsville, IN}
} \email{ansaldik@wabash.edu}

\address{
Department of Mathematics, Penn State Greater Allegheny,  McKeesport, PA
}
\email{kul20@psu.edu}

\address{Wu Wen-Tsun Key Laboratory of Mathematics of CAS and School of Mathematical Sciences, University of Science and Technology of China, Hefei, Anhui, 230026, P.R.~China}
\email{yhshen@ustc.edu.cn}


\begin{abstract}
    Given a monomial ideal in a polynomial ring over a field, we define the generalized Newton complementary dual of the given ideal. We show good properties of such duals including linear quotients and isomorphisms between the special fiber rings. We construct the cellular free resolutions of duals of strongly stable ideals generated in the same degree. When the base ideal is generated in degree two, we provide an explicit description of cellular free resolution of the dual of a compatible generalized stable ideal.
\end{abstract}

\maketitle

\section{Introduction}

Given a polynomial ring $R=\KK[x_1,\ldots,x_n]$ over a field $\KK$ and a graded ideal $I$ in $R$, one would like to understand various algebraic properties of the ideal. For instance, the Castelnuovo-Mumford regularity, the projective dimension, and the Cohen-Macaulayness of $R/I$, are of great importance. Finding the minimal free resolution of the ideal is the key to those properties. This has been an active area among commutative algebraists and algebraic geometers. Another central topic in algebraic geometry is the blow up algebras defined by the given ideal. 

When $I$ is a monomial ideal, one can associate $I$ with combinatorial objects such as (hyper)graphs, and use combinatorial methods to recover algebraic properties; see, for example, the surveys \cite{H} and \cite{MV}. However, describing the precise minimal free resolution of a squarefree monomial ideal is not easy; see, for instance, \cite{AHH},  \cite{HV} and \cite{Ho}. There are even fewer results on finding the minimal free resolution for non-squarefree monomial ideals. The first non-trivial class to consider is that of stable ideals, studied by Eliahou and Kervaire in \cite{EK}. When $I$ is a monomial ideal generated in the same degree, its special fiber ring is the associated toric ring. Hence understanding the defining equations of the special fiber ring of the ideal is very important. Villarreal \cite{V} found the explicit description of the defining equations of the special fiber ring of edge ideals, i.e., squarefree monomial ideal generated in degree two. There are some other work on this subject, but they are almost always centered around squarefree monomial ideals. 

The motivation of this work comes from the paper \cite{CN1} of Corso and Nagel, where they studied the specialization of generalized Ferrers graphs (see Definitions \ref{Spec} and \ref{genferdef}). They showed that every strongly stable ideal of degree two can be obtained via a specialization. The authors later explicitly described the minimal free resolutions of Ferrers ideals and the defining equations of the special fiber ring in \cite{CN2}. For this purpose, they used cellular resolutions as introduced by Bayer and Sturmfels in \cite{BS}. This construction provides a characteristic-free context. Note that the special fiber ring is a determinantal ring in this case. 

Meanwhile,  the Newton complementary duals of monomial ideals were first introduced by Costa and Simis in \cite{CS}. There, the dual operation was applied to study the rational maps between the base ideals and the dual ideals. In this work, we will extend
the dual operation to get the generalized Newton complementary duals of monomial ideals, as introduced in \Cref{dual}. The generalized (Newton complementary) dual operation is indeed a dual operation, since the double dual will bring back the base ideal to itself (\Cref{doubleDual}).  Properties of generalized Newton complementary duals were investigated in this work.

In Section 3, we establish an isomorphism of special fiber rings between the base ideal and the generalized dual ideal (\Cref{FiberIso}), generalizing the corresponding result of Costa and Simis. In other words, we will prove that the base ideal defines a birational map if and only if its generalized dual ideal defines a birational map.  We then focus on the generalized duals of monomial ideals that are related to classical Ferrers graphs. The ideals that we consider are duals of specializations of generalized Ferrers ideals. As a corollary of \Cref{FiberIso} and the work of Corso, Nagel, Petrovi\'{c}, and Yuen in \cite{CNPY}, we describe the special fiber rings of the generalized duals of specializations of generalized Ferrers ideals. In other words, we describe the toric rings associated to the generalized duals of such ideals (\Cref{specfibershape}). Those toric rings are Koszul normal Cohen-Macaulay domains. In particular, we establish a new class of Koszul ideals.

In Section 4, we focus on the basic properties of duals of monomial ideals generated in the same degree. Whence, the generalized duals are also generated in the same degree. This class of monomial ideals has the nice property that it is closed under the ideal products. Moreover, when the base ideals are stable of degree two or strongly stable, the generalized duals have linear quotients (\Cref{LinQuo}); in particular, their regularities coincide with the common degrees of the minimal monomial generators. We further notice that Newton complementary duals of an ideal defined by a bipartite graph is the Alexander dual of the edge ideal of the complement of the base bipartite graph (\Cref{AlexDual}). However, this property does not hold when the base ideal is not coming from a bipartite graph. 
The reason for this, as pointed out later in the work of Budd and Van Tuyl (\cite[Discussion 3.8]{BV}), is the close relation between the Alexander dual and the complementary dual of squarefree monomial ideals, which can be best seen from the point of view of simplicial complexes; see \Cref{not-agree}.

In Section 5, we find a cellular complex which supports the minimal free resolution for the generalized duals of strongly stable ideals generated in degree $d$ (\Cref{Borel}). Consequently, we can easily recover the Betti numbers and projective dimensions of such ideals (\Cref{Betti}). The construction involved is inspired by the works of Corso and Nagel \cite{CN1} and \cite{CN2} as well as the work of Mermin \cite{Mer}. To be more precise, we use the technique of iterated mapping cones for the proof. This provides a geometric description of the free resolution and hence is characteristic-free. One also notices that, in some sense, our results provide the dual version of the work of Dochtermann and Engstr\"{o}m in \cite{DE}  or Nagel and Reiner in  \cite{NR}, where they found cellular resolutions of the edge ideals of cointerval hypergraphs or squarefree strongly stable ideals generated in a fixed degree. It is worth pointing out that the generalized dual ideals are usually non-squarefree, whereas they focused on squarefree cases. 

In the last section, inspired by the work of Nagel and Reiner in \cite{NR}, we consider a type of monomial ideals of degree 2 that generalizes the concept of stable ideal. When the base ideal is compatible in the sense of \Cref{compat}, we can describe explicitly a cellular free resolution of the generalized dual. The construction uses the quasi-Borel move which is similar to the Borel move introduced in \cite{FMS}. And the proof uses the properties of linear quotients and iterated mapping cones (\Cref{shifted-stable-linear-quot} and \Cref{compatible-stable}). In the special case when the base ideal is stable, one has a very neat formula for the Betti numbers of the generalized dual ideal (\Cref{bettiStable}). This construction offers a quick and computationally efficient way to obtain the minimal free resolution from a graphical investigation.

\section{Preliminary}

Let $R = \KK[x_1,\ldots,x_n]$ be a polynomial ring over a field $\KK$.  We give $R$ a standard graded structure, where all variables have degree one.  We write $R_i$ for the $\KK$-vector space of homogeneous degree $i$ forms in $R$ so that $R = \bigoplus_{i \ge 0} R_i$.  We use the notation $R(-d)$ to denote a rank-one free module with the generator in degree $d$ so that $R(-d)_i = R_{i-d}$.  

Analogously, the ring $R$ is endowed with a multigraded ($\NN^n$-graded) structure. Whence, if $\bda\in\NN^n$, following \cite[Section 26]{P}, we will usually say ``multidegree $\bdx^{\bda}$'' instead of ``$\NN^n$-degree $\bda$''. Meanwhile, $R(\bdx^{\bda})=R(-\bda)$ means the free $R$-module with one generator in the multidegree $\bdx^{\bda}$.

Let $M$ be a finitely generated graded $R$-module.  We can compute the minimal graded free resolution of $M$: 
\[
0 \to \bigoplus_{j} R(-j)^{\beta_{pj}(M)} \to \cdots\to \bigoplus_{j} R(-j)^{\beta_{2j}(M)} \to \bigoplus_{j} R(-j)^{\beta_{1j}(M)}  \to \bigoplus_{j} R(-j)^{\beta_{0j}(M)}\to 0.
\]
The minimal graded free resolution of $M$ is unique up to isomorphism.  Hence, the numbers $\beta_{ij}(M)$, called the \textit{graded Betti numbers} of $M$, are invariants of $M$.  Two coarser invariants measuring the complexity of this resolution are the \textit{projective dimension} of $M$, denoted by $\pd_R(M)$, and the \textit{Castelnuovo-Mumford regularity} of $M$, denoted by $\reg(M)$.  They are defined as
\[
\pd_R(M) = \max\Set{i\mid\beta_{ij}(M) \neq 0 }
\]
and
\[ 
\reg(M) = \max\Set{j - i\mid\beta_{ij}(M) \neq 0},
\]
respectively.

Now, we introduce the generalized Newton complementary dual of a given monomial ideal $I$ in the ring $R$. The Newton complementary dual was first introduced by Costa and Simis in \cite{CS}. Here, $[n]$ with $n\in\NN$ stands for the set $\Set{1,2,\dots,n}$.

\begin{Definition}
    \label{dual}
    Let $G(I)=\{f_1, \ldots, f_\nu\}$ be the minimal monomial generating set of $I$. Let $\bda=(\bda(1),\bda(2),\dots,\bda(n))\in \NN^n$. If for each $f_i=\bdx^{\bdalpha_i}:=\prod_k x_k^{\bdalpha_i(k)}\in G(I)$ and $j\in [n]$, one has $\bdalpha_i(j)\le \bda(j)$, we say $I$ is \emph{$\bda$-determined}. The \emph{generalized Newton complementary dual} of $I$ determined by $\bda$ is the monomial ideal $\widehat{I}^{[\bda]}$ with $G(\widehat{I}^{[\bda]})=\Set{\widehat{f_i}:=\bdx^\bda/f_i|f_i\in G(I)}$. We simply call $\widehat{I}^{[\bda]}$ the $\bda$-dual of $I$. When the vector $\bda$ is clear from the context, we also call $\widehat{I}^{[\bda]}$ the generalized dual of $I$.
\end{Definition}

We illustrate the concept of the generalized dual with a simple example. 

\begin{Example}
    Consider the ideal $I =(x^3, x^2y^2, y^4) \subseteq \KK[x,y]$. Let $\bda=(5,6)$. The generalized Newton complementary dual of $I$ determined by $\bda$ is $\widehat{I}^{[\bda]}$ with $G(\widehat{I}^{[\bda]})=\{x^2y^6,x^3y^4,x^5y^2\}$.
\end{Example}

\begin{Remark}
    In \Cref{dual}, when $\bda(j)=\max_i\{\bdalpha_i(j)\}$ for each $j$, the ideal $\widehat{I}^{[\bda]}$ is exactly the Newton complementary dual of $I$ defined in \cite{CS}. We will write $\widehat{I}^{[\bdN]}$ for the Newton complementary dual of $I$.
\end{Remark}
    
\begin{Remark}
    \label{doubleDual}
    Two easy observations: 
    \begin{enumerate}[a]
        \item double applications of the generalized dual bring back the base ideal:$\widehat{\left(\widehat{I}  ^{[\mathbf{a}]} \right)}^{[\mathbf{a}]}=I$;
        \item when $I$ and $J$ are two $\bda$-determined monomial ideals generated in degrees $i$ and $j$ respectively, one has $\widehat{IJ}^{[2\bda]} = \widehat{I}^{[\bda]}\widehat{J}^{[\bda]}$ where $2\bda=(2\bda(1),2\bda(2),\dots,2\bda(n))$.
    \end{enumerate}
\end{Remark}

Next, we recall some of the definitions and theorems regarding the cellular resolution of a monomial ideal from \cite{MS}. One of the main goals of this work is to establish a cellular complex that provides the minimal free resolution of the $\bda$-dual of an ideal. This topic will be further investigated in later sections.

\begin{Definition}
    [{\cite[Section 4.1]{MS}}]
    A \emph{(polyhedral) cell complex} $X$ is a finite collection of finite polytopes (in $\mathbb{R}^n$) called the \emph{faces} or \emph{cells} of $X$, satisfying the following conditions.
    \begin{enumerate}[a]
        \item If $P \in X$ is a polytope in $X$ and $F$ is a face of $P$ then $F \in X$. 
        \item If $P, Q \in X$, then $P \cap Q$ is a common face of $P$ and $Q$. 
    \end{enumerate}
    The maximal faces are called \emph{facets}. The \emph{dimension} of $X$ is determined by the maximal dimension of its facets.
    When $Q\subset P$ are two faces, $Q$ is called a \emph{facet} of $P$ whenever $Q$ is one dimension less than $P$.
    A cell complex $X$ is \emph{labeled} if we can associate to each vertex a vector $\bdalpha_i \in \NN^n$. The induced label $\bdalpha_F$ of any face $F$ of $X$ is the exponent vector of  $\lcm \Set{ \bdx^{\bdalpha_i} \mid i \in F}$.
\end{Definition}

Since each vector $\bdalpha\in\NN^n$ can be identified with the monomial $\bdx^\bdalpha\in \KK[x_1,\dots,x_n]$, for simplicity, we will also say that the above face $F$ is labeled by the monomial $\bdx^{\bdalpha_F}=\lcm\Set{ \bdx^{\bdalpha_i} \mid i \in F}$.

Let $F_k(X)$ be the set of faces of $X$ of dimension $k$.  Note that the empty set is the unique $(-1)$-dimensional face. A cell complex $X$ has an \emph{incidence function} $\varepsilon$, where $\varepsilon(Q, P) \in \{1,-1\}$ if $Q$ is a facet of $P$. Note that the sign is determined by whether the orientation of $P$ induces the orientation of $Q$ where the orientation is determined by some ordering of the vertices.

Let $X$ be a cell complex of dimension $d$. The \emph{cellular free complex $\mathcal{F}_X$ supported on} $X$ is the complex of $\mathbb{N}^n$-graded $R$-modules
\begin{equation}
\mathcal{F}_X: \quad 0 \rightarrow  R^{F_d(X)} \xrightarrow{\partial_d} R^{F_{d-1}(X)} \xrightarrow{\partial_{d-1}} \cdots \xrightarrow{\partial_{2}} R^{F_{1}(X)}\xrightarrow{\partial_{1}} R^{F_{0}(X)} \xrightarrow{\partial_0} R=R^{F_{-1}(X)}  \rightarrow 0,
\label{FX-complex}
\end{equation}
where $R^{F_k(X)} := \bigoplus_{P \in F_k(X)} R(-{\bdalpha}_P)$, and the differential map $\partial_k$ is defined on the basis element $P$ of $R(-\bdalpha_P)$ in $R^{F_k(X)}$ as
\[
\partial(P) = \sum_{Q \text{ is a facet of } P} \varepsilon (Q,P) \bdx^{{\bdalpha}_P -{\bdalpha}_Q}Q. 
\]

We may consider the componentwise comparison partial order on $\mathbb{N}^n$ defined by $\bdalpha \leq \bdbeta$ whenever $\bdbeta -\bdalpha \in \mathbb{N}^n$. If $\bdbeta \in \NN^n$, we define a subcomplex $X_{\leq \bdbeta}$, namely the subcomplex of faces whose labels are less than or equal to $\bdbeta$. 

A common procedure to determine whether $\calF_X$ in \eqref{FX-complex} is a resolution is by applying the following criteria of Bayer and Sturmfels. This criteria is useful because it reduces the question of whether a cellular free complex is acyclic to a question of the geometry of the polyhedral cell complex.

\begin{Lemma} 
    [{\cite{BS}}]
    \label{BayerSturmfels}
    The complex $\mathcal F_X$ is a cellular resolution if and only if for each $\bdbeta$ the complex $X_{\leq \bdbeta}$ is acyclic over the base field $\KK$.
\end{Lemma}

However, we will take a different approach by applying iterated mapping cones.
Let us recall some of the basic constructions in \cite{P}.

\begin{Definition}
    [{\cite[Section 27]{P}}]
    \label{mapping-cone}
    Let $\varphi: (\bm{U},d)\to (\bm{U'},d')$ be a map of complexes of finitely generated $R$-modules. The map $\varphi$ is also called a \emph{comparison map}. The \emph{mapping cone} of $\varphi$ is the complex $\bm{W}$ with the differential $\partial$, defined as follows:
    \begin{align*}
        W_i & = U_{i-1}\oplus U_i' \text{ as a module}, \\
        \partial|_{U_{i-1}}&=-d+\varphi: U_{i-1}\to U_{i-2}\oplus U_{i-1}',\\
        \partial|_{U_i'} & = d': U_i'\to U_{i-1}',
    \end{align*}
    for each $i$. 
\end{Definition}

\begin{Remark}
    [{\cite[Construction 27.3 and the discussion before it]{P}}]
    Suppose that $\bm{U}$ and $\bm{U'}$ above are free resolutions of finitely generated modules $V$ and $V'$ respectively, while $\varphi:V\to V'$ is an injective homomorphism of modules. Then, there is a lifting of $\varphi$ to $\bm{U}\to \bm{U'}$, which will also be denoted by $\varphi$. The mapping cone of $\varphi$ provides a free resolution of the quotient module $V'/\varphi(V)$. We are interested in the case when monomial ideals $I=(m_1,\dots,m_v)$ and $I'=(m_1,\dots,m_{v-1})$, while $V'=R/I'$ and $V=R/(I':m_v)$. Notice that there is a short exact sequence here:
    \[
    0\to R/(I':m_v)\xrightarrow{\cdot m_v} R/I' \to R/I \to 0.
    \]
\end{Remark}

In later sections, we will work on duals relative to Ferrers ideals, stable ideals and strongly stables. We recall some definitions first. For a monomial $m\in R=\KK[x_1,\dots,x_n]$, we write $\supp(m)$ for the set $\Set{i\in [n]|\text{$x_i$ divides $m$}}$ and $\supp_1(m)$ for the subset $\supp(m)\setminus\{1\}$. We also write $\max(m)$ for $\max\supp(m)$. A monomial ideal $I$ is called \emph{stable} if for each monomial $m\in I$, for each $i<\max(m)$, one has $m x_i/x_{\max(m)}\in I$. The ideal $I$ is called \emph{strongly stable} if for each monomial $m\in I$, for each $x_i$ dividing $m$ and $j<i$, one has $mx_j/x_i\in I$. Such ideals are also called \emph{Borel ideals} in the literature. For both types, it is easy to see that when the ideal $I$ is generated in the same degree, it suffices to check the monomials in $G(I)$. Throughout this paper, we will assume that all the monomials in $G(I)$ have degree $d$. The definition below is defined in \cite{CN1} and we will use a similar concept on the duals later. 
\begin{Definition}
    [{\cite[Definition 3.1]{CN1}}]
    \label{Spec}
    Let  $S = \KK[x_1, \ldots, x_m, y_1, \ldots, y_n]$ be a polynomial ring over a field $\KK$ and $I$ be a monomial ideal in $S$. Let $\sigma:\{y_1, \ldots, y_n\} \rightarrow \{x_1, \ldots, x_k\}$ be a map  that sends $y_i$ to $x_i$ where $k = \max\{m,n\}$ and $x_{m+1}, \ldots, x_k$ are (possibly) additional variables. By abuse of notation, we use the same symbol to denote the substitution homomorphism $\sigma:S \rightarrow R= \KK[x_1, \ldots, x_k]$, given by $x_i \mapsto x_i$ and $y_i \mapsto \sigma(y_i)$. We call $\sigma$ a \emph{specialization map} and the monomial ideal $\overline I := \sigma (I) \subseteq R$ the \emph{specialization} of $I$. 
\end{Definition}

Here is an example of the specialization of an ideal.

\begin{Example}
    \label{less}
    Let $S = \KK[x_1, x_2, y_1, y_2, y_3]$. Consider the monomial ideal $I$ generated by  
    \[
    x_1y_1, x_1y_2, x_1y_3, x_2y_1, x_2y_2
    \]
    in $S$. The specialization of $I$ is the ideal 
    $$
    \overline I = (x_1^2, x_1x_2, x_1x_3, x_2^2)\subseteq \KK[x_1,x_2,x_3].
    $$
    Since the specialization map sends both $x_1y_2$ and $x_2y_1$ to the same element, $\overline I$ has $4$ minimal generators while $I$ has $5$.
\end{Example}

Motivated by this example, we consider the following class of monomial ideals of degree $2$.
\begin{Definition}
    [{\cite[Definition 3.4]{CN1}}] \label{genferdef} Let $\lambda = (\lambda_1, \ldots, \lambda_m)$ be a partition with $\lambda_i \in \Z^+$, and $\lambda_m \leq \lambda_{m-1} \leq \cdots \leq \lambda_1$. Let $\mu = (\mu_1, \ldots, \mu_m) \in \Z^m$ be a vector with
    $$
    0 \leq \mu_1 \leq \cdots \leq \mu_m <\lambda_m.
    $$
    Since $\lambda_m \leq \lambda_{m-1} \leq \cdots \leq \lambda_1$, in particular, $\mu_i < \lambda_i$. 
    The ideal
    $$
    I_{\lambda-\mu} := (x_iy_j \mid 1 \leq i \leq m,\, \mu_i < j \leq \lambda_i)  
    $$
    is called a \emph{generalized Ferrers ideal}.
\end{Definition}

When $\mu$ is the zero vector, we get back the original \emph{Ferrers ideal} $I_{\lambda}$. On the other hand, when $\mu_i \geq i-1$ for $i =1, \ldots, m$, the generalized Ferrers ideal and its specialization have the same number of minimal generators, since no colliding phenomenon as in \Cref{less} will ever happen. 
The common number of minimal generators is simply $(\lambda_1 + \cdots + \lambda_m) - (\mu_1 + \cdots + \mu_m)$. We will investigate similar patterns in Section 3 and 6.
Notice that the most interesting case is when $\mu_i = i-1$ for each $i$. In this situation, the specialization of the generalized Ferrers ideal is a strongly stable ideal, as observed in \cite{CN1}. 

\begin{Example}
    \label{443}
    Let $S = \KK[x_1, x_2, x_3, y_1, y_2, y_3, y_4]$ and $I$ be the  Ferrers ideal for $\lambda = (4,4,3)$, that is, 
    $$I_\lambda = (x_1y_1, x_1y_2, x_1y_3, x_1y_4, x_2y_1, x_2y_2, x_2y_3, x_2y_4, x_3y_1, x_3y_2, x_3y_3).$$
    Let $\mu_i = i-1$ for $i = 1, 2, 3$. 
    The generalized Ferrers ideal is
    $$I_{\lambda-\mu} = (x_1y_1, x_1y_2, x_1y_3, x_1y_4, x_2y_2, x_2y_3, x_2y_4, x_3y_3).$$
    Then the specialization map yields the ideal 
    $$\overline{I_{\lambda-\mu}}= (x_1^2, x_1x_2, x_1x_3, x_1x_4, x_2^2, x_2x_3, x_2x_4, x_3^2).$$
    Note that $\overline{I_{\lambda-\mu}}$ is a strongly stable ideal in $R = \KK[x_1, x_2, x_3, x_4]$. 
\end{Example}

\section{Toric rings associated to dual ideals}

Given $I=(f_1, \ldots, f_{\nu})$, a monomial ideal generated in the same degree, we can define a toric ring, $ \KK[f_1, \ldots, f_{\nu}]$,  which is isomorphic to the special fiber ring of $I$. Recall that the special fiber ring of $I$ is the subring $\mathcal F(I) = \KK[f_1t, \ldots, f_{\nu} t]\subset R[t]$ where $t$ is a new variable. Geometrically,  the special fiber ring $\mathcal F(I)$ is the homogeneous coordinate ring of the image of a map $\mathbb{P}^{n-1} \rightarrow \mathbb{P}^{{\nu}-1}$. There is a natural surjective map $\phi: \KK[T_1, \ldots, T_{\nu}] \rightarrow \KK[f_1t, \ldots, f_{\nu}t]$. Consequently, we have a short exact sequence
\[
0 \rightarrow J \rightarrow \KK[T_1, \ldots, T_{\nu}] \rightarrow \KK[f_1t, \ldots, f_{\nu}t] \rightarrow  0, 
\]
where $J=\ker(\phi)$ is generated by all forms $F(T_1, \ldots, T_{\nu})$ such that $F(f_1, \ldots, f_{\nu})  = 0$. Note that $J$ is graded. 
In this section, we work on finding the defining equations of $J$ where $I$ is the $\bda$-dual of a monomial ideal generated in the same degree. 

The following theorem shows that the special fiber rings of the generalized Newton complementary duals and the given monomial ideals are isomorphic. This is a straightforward generalization of Costa and Simis \cite[Lemma 1.7]{CS}.

\begin{Theorem}
    \label{FiberIso}
    Let $R = \KK[x_1, \ldots, x_n]$ be a polynomial ring in $n$ variables over a field $\KK$. Let $I$ be an $\bda$-determined monomial ideal such that $I$ is generated in the same degree. Then the special fiber ring of $I$ and that of $\widehat{I}^{[\bda]}$ are isomorphic:
    \[
    \mathcal F(I) = \KK[It] \cong \KK[\widehat{I}^{[\bda]} t ]=  \mathcal F(\widehat I^{[\bda]} ).
    \]
\end{Theorem}

We give a proof for completeness.

\begin{proof}  
    For the given vector ${\bda}$, we write $\widehat{I}$ instead of $\widehat{I}^{[\bda]}$ to simplify the notation.  Suppose $G(I) = \{f_1, \ldots, f_{\nu}\}$. 
    Let $J_r$ be the degree $r$ piece of the kernel ideal $J$ above. If ${\bdalpha} = (i_1, \ldots, i_r)$ is a non-decreasing sequences of integers in $[\nu]$, we write $T_{\bdalpha} = \prod_k T_{i_k}$ and $f_{\bdalpha} = \prod_k f_{i_k}$.
    By \cite{Taylor}, $J_r$ is generated by polynomials of the form
    \begin{equation}
        \label{relations}
        \Set{ T_{\bdalpha} - T_{\bdbeta}\mid {\bdalpha} = (i_1, \ldots, i_r) \text{ and } {\bdbeta} = (j_1, \ldots, j_r) \text{ with } f_{\bdalpha}=f_{\bdbeta}}.
    \end{equation}
    
    Since $\widehat{I}$ is again a monomial ideal, there is a surjective map $\psi: \KK[S_1, \ldots, S_{\nu}] \rightarrow \KK[\widehat{I}t]$ given by $\psi(S_i) = \widehat{f_i}t$.
    Let $J'$ be the kernel of $\psi$. Likewise, its degree $r$ piece $J'_r$ is generated by polynomials of the form
    \[
    \Set{ S_{\bdalpha} - S_{\bdbeta}\mid {\bdalpha} = (i_1, \ldots, i_r) \text{ and } {\bdbeta} = (j_1, \ldots, j_r) \text{ with } \widehat{f_{\bdalpha}}=\widehat{f_{\bdbeta}}},
    \]
    where $S_{\bdalpha} = \prod_k S_{i_k}$ and $\widehat{f}_{\bdalpha} = \prod_k \widehat{f}_{i_k}$. 
    One notice immediately that for ${\bdalpha} = (i_1, \ldots, i_r)$,
    \begin{equation}
        \label{fhatalpha}
        \widehat{f_{\bdalpha}} = \frac{\bdx^\bda}{f_{i_1}} \frac{\bdx^\bda}{f_{i_2}} \cdots \frac{\bdx^\bda}{f_{i_r}} = \frac{(\bdx^{\bda})^r}{f_{\bdalpha}}.
    \end{equation}

    To show that the two special fiber rings are isomorphic, we define the natural map 
    \begin{eqnarray*}
        w': \KK[T_1, \ldots, T_{\nu}] &\rightarrow & \KK[S_1, \ldots, S_{\nu}]\\
        T_i & \mapsto & S_i.
    \end{eqnarray*} 
    For each $h \in \ker(\phi)$, we may indeed assume that $h=T_{\bdalpha} -  T_{\bdbeta}\in J_r$ as in \eqref{relations}. Now, 
    using (\ref{fhatalpha}), we have
    \[
    \psi(w'(h)) = \psi(S_{\bdalpha} - S_{\bdbeta}) =\widehat{f_{\bdalpha}} - \widehat{f_{\bdbeta}}
    =\frac{(\bdx^{\bda})^r}{f_{\bdalpha}} - \frac{(\bdx^{\bda})^r}{f_{\bdbeta}}=0,
    \]
    since $f_{\bdalpha} = f_{\bdbeta}$. 
    Thus, $w'(h) \in \ker \psi = J'$ and in turn $w'(J) \subseteq J$. This also induces a map
    \begin{eqnarray*}w : \KK[It] &\rightarrow & \KK[\widehat{I}t]\\
        f_it & \mapsto & \widehat{f_i} t,
    \end{eqnarray*}
    which is well-defined.

    Since $\widehat{\widehat{I}}=I$, we can define similar maps $v: \KK[\widehat{I}t] \rightarrow \KK[\widehat{\widehat{I}}t] = \KK[It]$ and $v': \KK[S_1, \ldots, S_{\nu}] \rightarrow \KK[T_1, \ldots, T_{\nu}]$ as above.    By the same argument, we have that $v'(J') \subseteq J$.  Since $w'$ and $v'$ are obviously inverse maps, we have $J = v'(w'(J)) \subseteq v'(J') \subseteq J$. Thus $v'(J') = J$.  Similarly, $w'(J) = J'$. Thus we have
    \[
    \KK[It] \cong \frac{\KK[T_1, \ldots, T_{\nu}]}{J} \cong \frac{\KK[S_1, \ldots, S_{\nu}]}{J'}  \cong \KK[\widehat I t].
    \qedhere
    \]
\end{proof}

The work of Corso, Nagel, Petrovi\`c, and Yuen \cite{CNPY} considered the special fiber ring $\mathcal{F}(\overline{I_{\lambda-\mu}})$ of the specialization of the generalized Ferrers ideal $I_{\lambda-\mu}$ with $\lambda=(\lambda_1,\ldots,\lambda_m )$ and $\mu=(\mu_1,\dots,\mu_m)$, such that $\mu_i\ge i-1$ for each $i$.
For this purpose,
consider the polynomial ring 
\[
\KK[\mathbf{T}_{\lambda}]:=\KK[T_{ij}\mid x_iy_j\in I_{\lambda-\mu}] = \KK[T_{ij}\mid 1\le i\le m,\, \mu_i<j\le \lambda_i].
\]
Let $n=\lambda_1$.  Thus, we can think of $\mathbf{T}_{\lambda}$ as an $m\times n$ matrix with the variable $T_{ij}$ as the $(i,j)$ entry, when $x_iy_j\in I$; otherwise, the entry is $0$. The \emph{symmetrized matrix} $\mathbf{S}_{\lambda}$ is the $n\times n$ matrix obtained by reflecting $\mathbf{T}_{\lambda}$ along the main diagonal. Notice that $n\ge m$ and $\mathbf{T}_{\lambda}$ is upper-triangular.

\begin{Theorem}
    [{\cite[Theorem 4.2 and Proposition 4.1]{CNPY}}]
    Let $I \subseteq \KK[x_1, \ldots, x_n]$ be a specialization of generalized Ferrers ideal.  The special fiber ring of $I$ is a determinantal ring arising from the $2\times 2$ minors of a symmetric matrix. More precisely, there is a graded isomorphism 
    $$
    \mathcal F(I) \cong \KK[\mathbf{T}_{\lambda}]/I_2(\mathbf{S}_\lambda),
    $$ 
    by using notations above.  Furthermore, the ring $\mathcal F(I)$ is a Koszul normal Cohen-Macaulay domain of Krull dimension $n$.  

\end{Theorem}

By the above theorem and \Cref{FiberIso}, we can describe the special fiber rings of generalized Newton complementary duals of specialization of generalized Ferrers ideals.

\begin{Corollary}
    \label{specfibershape}
    Let $I$ be an $\bda$-determined specialization of generalized Ferrers ideal, and let $\widehat{I}^{[\bda]}$ be the $\bda$-dual of $I$. The special fiber ring of $\widehat{I}^{[\bda]}$ is a determinantal ring arising from the $2\times 2$ minors of a symmetric matrix. More precisely, there is a graded isomorphism 
    $$\mathcal F(\widehat I^{[\bda]}) \cong \KK[\mathbf{T}_{\lambda}]/I_2(\mathbf{S}_\lambda),$$
    by using notations above.  Furthermore, $\mathcal F(\widehat I^{[\bda]})$ is a Koszul normal Cohen-Macaulay domain of Krull dimension $n$.  
\end{Corollary}

\section{Properties of the generalized Newton complementary dual}

In this section, we provide additional nice properties of the generalized Newton complementary dual.  Within this section, we will always consider the $\emph{co-lexicographic}$ total order $\prec$ on the monomials in $R$ of degree $d$: we will say $\bdx^{\bdalpha} \prec \bdx^{\bdbeta}$ if there is a $k\in [n]$ such that $\bdalpha(k)<\bdbeta(k)$, while $\bdalpha(j)=\bdbeta(j)$ for $k+1\le j\le n$. For our monomial ideal $I$ generated in degree $d$, we will always assume that $G(I)=\Set{f_1\prec f_2 \prec \cdots \prec f_{\nu}}$. The following observation is easy to verify.

\begin{Lemma}
    Let $I$ be a (strongly) stable ideal generated in the same degree as above. Then the sub-ideal $I':=(f_1,\dots,f_{\nu-1})$ is also (strongly) stable.
\end{Lemma}

\begin{Theorem}
    \label{LinQuo}
    Let $R=\KK[x_1,\dots,x_n]$ be a polynomial ring in $n$ variables over a field $\KK$. Let $I$ be an $\bda$-determined monomial ideal such that $I$ is generated in the same degree $d$. If $I$ is stable with $d=2$, or $I$ is strongly stable, then the $\bda$-dual $\widehat{I}^{[\bda]}$ has linear quotients. In particular, $\widehat{I}^{[\bda]}$ has a linear resolution.
\end{Theorem}

\begin{proof}
    Without loss of generality, we assume that $d\ge 2$.  Write $I':=(f_1\prec \cdots \prec f_{\nu-1})$ as above. By induction, it suffices to show that the colon ideal 
    \[
    J:=\widehat{I'}^{[\bda]}:\widehat{I}^{[\bda]}=\widehat{I'}^{[\bda]}:\widehat{f_{\nu}}
    \] 
    is linear. Since $\deg(f_{\nu})=d$, we may write $f_{\nu}=x_{t_1}x_{t_2}\cdots x_{t_d}$ with $t_1\le t_2\le \cdots \le t_d=\max(f_{\nu})$.
    Write 
    \[
    X_\nu:=\Set{x_j| \text{there exists $i<j$ such that $f_{\nu}x_i/x_j\in G(I)$}}.
    \]
    Notice that $I$ is stable. Thus $f_{\nu} x_{t_d-1}/x_{t_d}\in G(I)$. This means that $x_{t_d}\in X_{\nu}\ne \varnothing$. We want to show that $J=(X_{\nu})$. This, in particular, implies that $J$ is linear.

    Notice that for each $x_j\in X_{\nu}$ with $i<j$ and $f':=f_{\nu}x_i/x_j\in G(I')$, it can be translated into $x_{j}\widehat{f_{\nu}}=x_{i}\widehat{f'}$.  Thus $x_{j}\in J$. In turn, $(X_{\nu})\subseteq J$.

    Now, take a minimal monomial generator $y$ of $J$. By definition, $y\widehat{f_{\nu}}=w\widehat{f_k}$ for some integer $k<\nu$ and some monomial $w\in R$. This is equivalent to saying that $yf_k=wf_{\nu}$. By the minimality of $y$, one has $\gcd(y,w)=1$. Thus, $y$ divides $f_{\nu}$. 
    \begin{enumerate}[i]
        \item Suppose that $I$ is stable with $d=2$. If $y$ is not linear, this reduces to $y=f_{\nu}=x_{t_1}x_{t_2}$. But we already have $x_{t_2}\in J$. This contradicts the minimality of $y$.
        \item Suppose that $I$ is strongly stable with $d\ge 2$. Whence, 
            $\supp(X_{\nu})=\supp_1(f_{\nu})$.
            As $y\ne x_1$ divides $f_{\nu}$, for any $j\in \supp_1(y)\subseteq \supp_1(f_{\nu})$, we have $x_j\in J$. Thus, by the minimality of $y$, we have $y=x_j$, unless $y=x_1^t$ for some $t\ge 1$. In the latter case, as $yf_k=wf_{\nu}$ and $\gcd(y,w)=1$, we will have $f_k\succ f_{\nu}$, which is a contradiction.
    \end{enumerate}
    Therefore, we have shown that $J\subseteq (X_{\nu})$.

    The ``in particular'' part follows from \cite[Proposition 8.2.1]{HH}.
\end{proof}

The following example shows that for ideals generated by elements of degree greater than $2$, the strongly stable condition is necessary.

\begin{Example}
    Let $I$ be the ideal generated by
    \[
    x_1^3, x_1^2 x_2, x_1 x_2^2, x_2^3, x_1^2 x_3, x_1 x_2 x_3, x_2^2 x_3, x_1 x_3^2, x_2 x_3^2, x_3^3, x_1 x_2 x_4, x_3^2 x_4
    \]
    in $R=\QQ[x_1,x_2,x_3,x_4]$. It is not difficult to verify that $I$ is stable, but not strongly stable. The Newton complementary dual of  $I$ does not have linear quotients. Indeed, computation by \texttt{Macaulay2} \cite{M2} suggests that $\widehat{I}^{[\bdN]}$ does not have a linear resolution.
\end{Example}

Next, we examine the $\bda$-duals of edge ideals associated to bipartite graphs. More precisely, we focus on the Newton complementary dual $\widehat{I}^{[\bdN]}$ when $I$ is such an edge ideal; whence, $\bdx^{\bda}=\lcm{G(I)}$. For simplicity, we say Newton-dual instead of Newton complementary dual.  We begin by recalling several definitions and results about edge ideals and graphs, and then consider a connection among squarefree monomial ideals, Alexander Duals and Newton-duals.  

Let $R = \KK[x_1, \ldots, x_n]$ be the polynomial ring on $n$ variables. Suppose that $G$ is a finite simple graph (that is, a graph that does not have loops or multiple edges) with vertices labeled by $x_1, \ldots, x_n$.
The \emph{edge ideal} of $G$, denoted by $I(G)$, is the ideal of $R$ generated by the squarefree monomials $x_ix_j$ such that $\{x_i, x_j\}$ is an edge of $G$.  This gives a one-to-one correspondence between finite simple graphs and squarefree monomial ideals generated in degree $2$. 

The \emph{complement} of a graph $G$, is the graph with identical vertex set such that its edge set contains the edge $\{x_i, x_j\}$ if and only if $\{x_i, x_j\}$ is not an edge of $G$. 

For a subset $\sigma \subseteq [n]$, let $\bdx^\sigma = \displaystyle \prod_{i \in \sigma} x_i$. Note that any squarefree monomial in $\KK[x_1, \ldots, x_n]$ can be written in this way.  Let $\p_\sigma$ be the prime ideal $\p_\sigma = (x_i \mid i \in \sigma)$. For any squarefree monomial ideal $I = (\bdx^{ \sigma_1 },  \ldots, \bdx^{\sigma_r}) \subset \KK[x_1, \ldots, x_n]$, the \emph{Alexander dual} of $I$ is 
$$I^\star = \p_{\sigma_1} \cap \cdots \cap \p_{\sigma_r}.$$ 

In \Cref{AlexDual}, we will study the relation between the Alexander dual and the Newton-dual for bipartite graphs. 
Let $G$ be a bipartite graph with respect to vertex partition $X\sqcup Y$. Let $I=I(G)$ be the edge ideal. To remove isolated vertices, let
\[
X_I=\Set{x\in X\mid \deg_G(x)\ge 1},
\]
and we similarly define $Y_I$. Let $G^c$ be the complement graph of $G|_{X_I\sqcup Y_I}$ with respect to the vertex set $X_I\sqcup Y_I$. We may think of it as the \emph{essential complement} of $G$. 

\begin{Proposition}
    \label{AlexDual}
    Let $G$ be a finite bipartite graph corresponding to the vertex partition $X\sqcup Y$, and $I(G)$ the associated edge ideal in the polynomial ring $R = \KK[x,y\mid x\in X, y\in Y]$. Then 
    \begin{equation}
        \widehat{I(G)}^{[\bdN]}=(I(G^c))^\star.  
        \label{eq-AD}
    \end{equation}
\end{Proposition}

\begin{proof}
  The proof is a case-by-case checking using the elementary ideal intersection properties. 
\end{proof}

%

\begin{Remark}
    \label{not-agree}
    It is not difficult to see that the equality in \eqref{eq-AD} does not hold for general finite simple graphs. On the other hand, the Alexander dual and the complementary dual are closely related, at least for squarefree monomial ideals. The best framework to see this, as pointed out later in the work of Budd and Van Tuyl (\cite[Discussion 3.8]{BV}), is the ideals associated to simplicial complexes: the complementary dual of a squarefree ideal $I$ is the non-face ideal of the Alexander dual of the face complex of $I$.
\end{Remark}

\section{Cellular resolutions for duals of strongly stable ideals}

In this section, the ideal $I$ will always be assumed to be a strongly stable ideal in $R=\KK[x_1,\dots,x_n]$ of degree $d$. Every monomial in $R$ will be identified with its exponent in $\mathbb{N}^n$. By abuse of notation, monomials and their exponents will be treated interchangeably.

Let $m\in G(I)$ be a monomial and suppose that $\sigma\subseteq \supp_1(m)\subseteq [n]$. We write $m\to \sigma$ for the monomial 
\[
m \prod_{i\in \sigma}(x_{i-1}/x_i).
\]
As $I$ is strongly stable, the monomial $m\to \sigma$ still belongs to $G(I)$. We denote by $C(m,\sigma)$ the convex hull of $\Set{m\to \tau |\tau\subseteq \sigma}$ in $\mathbb{R}^n$, and let $X_I=\Set{C(m,\sigma)\mid m\in G(I), \sigma\subseteq \supp_1(m)}$.

\begin{Lemma}
    $X_I$ is a polyhedral cell complex.
    \label{Lemma-XI}
\end{Lemma}

\begin{proof}
    Write $\bde_1,\dots,\bde_n$ for the canonical bases of $\mathbb{R}^n$, and consider the non-degenerate linear map $L$ sending $\bde_i$ to $\bde_i':=\sum_{j=1}^i\bde_j$. Then, geometrically, the two vectors $L(m\to \sigma)$ and $L(m)$ differ by $\sum_{i\in \sigma}\bde_i$. 
    We observe the following facts.
    \begin{enumerate}
        \item The image $L(C(m,\sigma))$ is an $|\sigma|$-dimensional face of a unit cube in $\mathbb{R}^n$ with all corner vertices being lattice points in $\NN^n$.
        \item $L(m)$ is the unique largest point in $L(C(m,\sigma))$ with respect to the componentwise-comparison partial order in $\mathbb{N}^n$. In particular, $m$ and $\sigma$ uniquely determines $L(C(m,\sigma))$ and in turn $C(m,\sigma)$.
        \item The intersection of two admissible $L(C(m_1,\sigma_1))$ and $L(C(m_2,\sigma_2))$ still takes the form of $L(C(m,\sigma))$. In turn, the intersection of $C(m_1,\sigma_1)$ and $C(m_2,\sigma_2)$ still takes the form $C(m,\sigma)$.
    \end{enumerate}
    Therefore, $X_I$ is a polyhedral cell complex.
\end{proof}

Suppose that $I$ is $\bda$-determined.  To make the desired cellular resolution for $\widehat{I}^{[\bda]}$, we still need some preparations.
\begin{enumerate}[a]
    \item We label this complex. Each $0$-cell corresponding to the monomial $f\in G(I)$ will be labeled by $\widehat{f}=\bdx^{\bda}/f\in R$. Consequently, the face $C(f,\sigma)$ will be labeled by $\widehat{f}\bdx^{\sigma}$ where $\bdx^{\sigma}$ denotes $\prod_{j\in\sigma}x_j$.
    \item We give explicitly the incidence function $\varepsilon$ for $X_I$. Let $F_1=C(f,\sigma)$ be an $|\sigma|$-cell and $F_2$ an $(|\sigma|-1)$-dimensional face of $F_1$.
        \begin{enumerate}[a]
            \item If $F_2$ takes the form $C(f,\tau)$, then we write $\sigma=\Set{i_1<i_2<\cdots<i_k}$ and $\tau=\Set{i_1,\dots,i_{j-1},i_{j+1},\dots,i_k}$. We choose $\varepsilon(F_2,F_1)=(-1)^{j-1}$.
            \item Otherwise, there is a unique face $F_3$ of $F_1$ that is parallel to $F_2$ with the same dimension. We choose $\varepsilon(F_2,F_1)=-\varepsilon(F_3,F_1)$.
        \end{enumerate}
        One checks with ease that $\varepsilon$ is indeed an incidence function.
\end{enumerate}

Now, the labeled $X_I$ gives a cellular complex $\calF_{X_I}$.

\begin{Example} 
    \label{Example-XI}
    Let $I\subseteq \QQ[x_1,x_2,x_3,x_4]$ be the minimal strongly stable ideal that contains the monomial $x_2x_3x_4$. Thus, it has monomial generators
    \begin{align*}
    f_1=x_1^3, f_2=x_1^2x_2, f_3=x_1x_2^2, f_4=x_2^3, f_5=x_1^2x_3, f_6=x_1x_2x_3, f_7=x_2^2x_3, f_8=x_1x_3^2, \\
    f_9=x_2x_3^2, f_{10}=x_1^2x_4, f_{11}=x_1x_2x_4, f_{12}=x_2^2x_4, f_{13}=x_1x_3x_4, f_{14}=x_2x_3x_4.
\end{align*}
    Its Newton-dual is  
   \[
   \widehat{I}^{[\bdN]}=(\widehat{f_1}=x_2^3x_3^2x_4,\widehat{f_2}=x_1x_2^2x_3^2x_4, \dots ,\widehat{f_{14}}=x_1^3x_2^2x_3).
   \]
   The complex $X_I$ can be visualized as in \Cref{BorelXI} and some of the labels are omitted for a better reading. Notice that those $1$-cells (edges) that point north-eastward, north-westward or southward, are labeled as $\widehat{f_i}x_2$, $\widehat{f_i}x_3$ or $\widehat{f_i}x_4$, respectively. Similarly, all the $2$-cells (squares) are labeled as $\widehat{f_i}x_3x_4$, $\widehat{f_i}x_2x_3$ or $\widehat{f_i}x_2x_4$, depending on the position of the $2$-cell. Finally the only $3$-cell (cube) is labeled as $\widehat{f_{14}}x_2x_3x_4=\lcm(G(I))$ since we take $\bda=\lcm(G(I))$. One can easily read from the picture that there are $14$ vertices, $21$ edges, $9$ squares and $1$ cube. Meanwhile, one can check with \texttt{Macaulay2} \cite{M2} that the Betti numbers of $\widehat{I}^{[\bdN]}$ are $14$, $21$, $9$ and $1$ respectively.

     \begin{figure}[h]
        \begin{center}
            \begin{tikzpicture} [thick, scale=1.4, every node/.style={scale=0.8}]]
                \shade [shading=ball, ball color=black]  (0,0) circle (.04) node [above left] {\scriptsize$\widehat{f_1}$};
                \shade [shading=ball, ball color=black]  (-1,-1) circle (.04) node [above left] {\scriptsize$\widehat{f_2}$};
                \shade [shading=ball, ball color=black]  (-2,-2) circle (.04) node [above left] {\scriptsize$\widehat{f_3}$};
                \shade [shading=ball, ball color=black]  (-3,-3) circle (.04) node [above left] {\scriptsize$\widehat{f_4}$};
                \shade [shading=ball, ball color=black]  (0.2,-1.2) circle (.04) node [above right] {\scriptsize$\widehat{f_5}$};
                \shade [shading=ball, ball color=black]  (-0.8,-2.2) circle (.04) node [above right] {\scriptsize$\widehat{f_6}$};
                \shade [shading=ball, ball color=black]  (-1.8,-3.2) circle (.04) node [below left] {\scriptsize$\widehat{f_7}$};
                \shade [shading=ball, ball color=black]  (0.4,-2.4) circle (.04) node [above right] {\scriptsize$\widehat{f_8}$};
                \shade [shading=ball, ball color=black]  (-0.6,-3.4) circle (.04) node [below right] {\scriptsize$\widehat{f_9}$};
                \shade [shading=ball, ball color=black]  (0.2,-0.1) circle (.04) node [above right] {\scriptsize$\widehat{f_{10}}$};
                \shade [shading=ball, ball color=black]  (-0.8,-1.1) circle (.04) node [below right] {\scriptsize$\widehat{f_{11}}$};
                \shade [shading=ball, ball color=black]  (-1.8,-2.1) circle (.04) node [below right] {\scriptsize$\widehat{f_{12}}$};
                \shade [shading=ball, ball color=black]  (0.4,-1.3) circle (.04) node [above right] {\scriptsize$\widehat{f_{13}}$};
                \shade [shading=ball, ball color=black]  (-0.6,-2.3) circle (.04) node [below left] {\scriptsize$\widehat{f_{14}}$};

                \draw[thick, directed] (-1,-1) -- (0,0) node[above left, pos=0.5] {\scriptsize$\widehat{f_{2}}x_2$};
                \draw[thick, directed] (-2,-2) -- (-1,-1) node[above left, pos=0.5] {\scriptsize$\widehat{f_{3}}x_2$};
                \draw[thick, directed] (-3,-3) -- (-2,-2) node[above left, pos=0.5] {\scriptsize$\widehat{f_{4}}x_2$};
                \draw[thick, dotted,  directed] (-0.8,-2.2) -- (0.2,-1.2);
                \draw[thick, dotted,directed] (-1.8,-3.2) -- (-0.8,-2.2);
                \draw[thick, directed] (-0.6,-3.4) -- (0.4,-2.4) node[below right, pos=0.5] {\scriptsize$\widehat{f_{9}}x_2$};
                \draw[thick, dotted, directed] (0.2,-1.2) -- (-1,-1);
                \draw[thick, dotted, directed] (-0.8,-2.2) -- (-2,-2);
                \draw[thick, directed] (-1.8,-3.2) -- (-3,-3) node[below, pos=0.5] {\scriptsize$\widehat{f_{7}}x_3$};
                \draw[thick, dotted, directed] (0.4,-2.4) -- (-0.8,-2.2);
                \draw[thick, directed] (-0.6,-3.4) -- (-1.8,-3.2) node[below, pos=0.5] {\scriptsize$\widehat{f_{9}}x_3$};
                \draw[thick, directed] (-1.8,-2.1) -- (-0.8,-1.1) node[right, pos=0.3] {\scriptsize$\widehat{f_{12}}x_2$};
                \draw[thick, directed] (-1.8,-2.1) -- (-1.8,-3.2) node[left, pos=0.5] {\scriptsize$\widehat{f_{12}}x_4$};

                \draw[thick, directed] (-1.8,-2.1) -- (-1.8,-3.2);
                \draw[thick, directed] (-0.6,-2.3) -- (0.4,-1.3) node[right, pos=0.4] {\scriptsize$\widehat{f_{14}}x_2$};
                \draw[thick, directed] (-0.6,-2.3) -- (-0.6,-3.4) node[right, pos=0.3] {\scriptsize$\widehat{f_{14}}x_4$};
                \draw[thick, directed] (-0.6,-2.3) -- (-1.8,-2.1);
                \draw[thick, dotted,  directed] (-0.8,-1.1) -- (-0.8,-2.2);
                \draw[thick, directed] (-0.8,-1.1) -- (0.2,-0.1);
                \draw[thick, directed] (0.4,-1.3) -- (-0.8,-1.1);
                \draw[thick, directed] (0.2,-0.1) -- (0.2,-1.2) node[right, pos=0.4] {\scriptsize$\widehat{f_{10}}x_4$};
                \draw[thick, directed] (0.4,-1.3) -- (0.4,-2.4) node[right, pos=0.5] {\scriptsize$\widehat{f_{13}}x_4$};
                
                \path[fill=black,fill opacity=0.2]    (-0.8,-1.1)--(-1.8,-2.1)--(-1.8,-3.2)--(-0.6,-3.4)--(0.4,-2.4)--(0.4,-1.3)--cycle;
            \end{tikzpicture}
            \caption{$X_I$} \label{BorelXI}
        \end{center}
    \end{figure}
\end{Example}

Partially following the convention in \cite[Section 26]{P}, we write $R(\omega,\bdx^{\bdalpha})$ for the free $R$-module with one generator in multidegree $\bdx^{\bdalpha}$, that corresponds to the cell $\omega \in X_I$.

\begin{Theorem}
    \label{Borel}
    Let $I$ be an $\bda$-determined strongly stable ideal in $R=\KK[x_1,\dots,x_n]$, generated in degree $d$.
    Then the complex $\calF_{X_I}$ provides a minimal free resolution for $R/\widehat{I}^{[\bda]}$.
\end{Theorem}

\begin{proof}
    Let $G(I)=\Set{f_1\prec f_2\prec \cdots \prec f_{\nu}}$ and $I'=(f_1,\dots,f_{\nu-1})$, as in the proof for \Cref{LinQuo}. We prove by induction on the index $\nu$. Notice that $\Set{C(f_{\nu},\sigma)|\sigma \subseteq \supp_1(f_{\nu})}$ are the extra cells which $f_{\nu}$ contributes to $X_{I}$ compared to $X_{I'}$. Meanwhile, all other faces in $C(f_{i},\supp_1(f_{i}))$ for $i<\nu$ already lie in $X_{I'}$.

    Suppose that $\supp_1(f_{\nu})=\Set{{i_1}<i_2<\cdots<i_k}$. The Koszul complex $\bdK=\bdK(x_{i_1},\dots,x_{i_k})$ provides a multigraded minimal free resolution for $R/(x_{i_1},\dots,x_{i_k})$. Notice that 
    $K_j=\bigoplus_{\sigma} R(\bdx^{\sigma})$,
    where $\sigma$ ranges over all subsets of $\supp_1(f_{\nu})$ with $|\sigma|=j$. The restriction of the differential map $d_j:K_j\to K_{j-1}$ to the direct summands  $R(\bdx^{\sigma})$ and $R(\bdx^{\tau})$ is the multiplication by $(-1)^{t-1} x_{w_t}$, when $\sigma=\Set{w_1<w_2<\cdots<w_j}$ and $\tau=\Set{w_1<\cdots<w_{t-1}<w_{t+1}<\cdots<w_j}$; otherwise, it is $0$. 

    Now we shift the internal multigrading of $\bdK$ by $\widehat{f_{\nu}}=\bdx^{\bda}/f_{\nu}$ to get $\bdK'$. Thus, above $K_j$ is modified to $K_j'=\bigoplus_{\sigma} R(\bdx^{\sigma}\widehat{f_{\nu}})$. Notice that $R(\bdx^{\sigma}\widehat{f_{\nu}})$ corresponds to $R(C(f_{\nu},\sigma),\bdx^{\sigma}\widehat{f_\nu})$ in $\calF_{X_I}$.  And the differential maps agrees with the expected differentials in $\calF_{X_I}$ up to a shift.

    Next, we consider the comparison map $\varphi$ from $-\bdK'[1]$ to $\calF_{X_{I'}}$.
    Here, $[1]$ indicates a shift in homological degree, and the negative sign means replacing each differential map by its negative.
    For $\varphi_j:K'_{j+1}\to \calF(X_{I'})_j$, we look at the restriction to the direct summands $R(C(f_{\nu},\sigma),\bdx^{\sigma}\widehat{f_{\nu}})$ and $R(C(f',\tau),\bdx^{\tau}\widehat{f'})$. Here, $|\tau|+1=|\sigma|=j$.
    \begin{enumerate}[a]
        \item If $C(f',\tau)$ is not a face of $C(f_{\nu},\sigma)$, this restriction is $0$.
        \item Otherwise, $\tau\subset \sigma$. Write $\omega=\sigma\setminus\tau$. Then $f'$ is $f_{\nu}\to \omega$. Now, we define the restriction of $\varphi_j$ here to be the multiplication by  $\varepsilon(C(f_{\nu}\to \omega,\tau),C(f_{\nu},\sigma))\cdot x_\omega$, which agrees with the corresponding differential map in $\calF_{I_X}$.
    \end{enumerate}
    To verify that $\varphi$ is a valid comparison map, i.e., the corresponding diagrams commute, it suffices to notice that $\varepsilon$ is an incidence function, as we have altered the signs of the differential maps to get $-\bdK'[1]$. 

    Finally, the mapping cone of $\varphi$ gives a free resolution for $R/(\widehat{I'}^{[\bda]},\widehat{f_{\nu}})=R/\widehat{I}^{[\bda]}$, as $\widehat{I'}^{[\bda]}:\widehat{f_{\nu}}=\supp_1(f_{\nu})$ by \Cref{LinQuo}. This resolution is minimal by checking the differentials directly. It also agrees with the expected polyhedral cell resolution $\calF_{X_I}$.
\end{proof}

\begin{remark}
    \label{Betti}
    With assumptions as in \Cref{Borel}, let $G(I)=\Set{f_1,f_2,\ldots ,f_\nu}$. If we write $r_k=|\supp_1(f_k)|$ for each $k\in[\nu]$, then
    \[
    \beta_i\left( \widehat{I}^{[\bda]} \right) = \sum_{k=1}^{\nu} \binom{r_k}{i}
    \]
    by \Cref{LinQuo} and \cite[Corollary 8.2.2]{HH}. In particular, $\pd(\widehat{I}^{[\bda]})=\max\Set{r_1,r_2,\dots,r_\nu}$.
\end{remark}

\section{Resolutions of duals of stable ideals of degree 2}
In this section, the base ideal we consider is a type of monomial ideal of degree $2$ that generalizes the concept of stable ideal.  
Partially following \cite[Definition 2.1]{NR}, we will call  the set of lattice points
\[
\Set{(i,j)\in \ZZ_+^2\mid 1\le i\le j}
\]
 the \emph{shifted quadrant}. Within this framework, a \emph{shifted partition} 
\[
\lambda-\mu:=(\lambda_1,\lambda_2,\dots,\lambda_h;\;\mu_1,\mu_2,\dots,\mu_h)
\]
will satisfy 
$i-1\le \mu_i< \lambda_i$ for $i\in [h]$. 
Corresponding to this shifted partition, one has a \emph{shifted quasi-Ferrers diagram}
\[
D_{\lambda-\mu}=\Set{(i,j)\in \ZZ_+^2\mid \mu_i<j\le \lambda_i}
\]
that lies entirely in the shifted quadrant. Note that $\lambda_i$ does  not need to be greater than $\lambda_j$ when $i<j$ as in the classical Ferrers diagram (see \Cref{quasi-diag-exam}). 

The operation $c_1=(i_1,j_1)\to c_2=(i_2,j_2)$ within the shifted quadrant is called a \emph{quasi-Borel move} if either $i_1=i_2$ or $j_1=j_2$. The \emph{geometric length} of this move is simply the Euclidean distance between these two lattice points; in this case, it is $\max(|i_1-i_2|,|j_1-j_2|)$. A quasi-Borel move $c\to c'$ is \emph{minimal} (with respect to the given shifted quasi-Ferrers diagram $D$) if there is no other lattice point in $D$ that lies on the line segment from $c$ to $c'$. Of course, $c\to c'$ is a minimal quasi-Borel move if and only if $c'\to c$ is so.

A \emph{quasi-Borel walk} of length $t$ (with respect to $D$) is a concatenation of minimal $t$ quasi-Borel moves: $c_0\to c_1 \to \cdots \to c_t$ with $c_i\in D$. 
The shifted quasi-Ferrers diagram $D$ is called \emph{connected} if for every pair of lattice points $c,c'\in D$, there is a quasi-Borel walk $c=c_0\to c_1 \to \cdots \to c_t=c'$, such that each quasi-Borel move has geometric length $1$.  
Since each row of the shifted quasi-Ferrers diagram $D$ is connected,
this condition simply means for each $i\in[h-1]$, one has
\[
\Set{j\in \ZZ_+\mid \mu_i<j\le \lambda_i} \cap
\Set{j\in \ZZ_+\mid \mu_{i+1}<j\le \lambda_{i+1}} \ne \varnothing.
\]

Notice that each lattice point $(i,j)$ in the shifted quadrant corresponds uniquely to the monomial $x_ix_j$ of degree $2$ in the polynomial ring $R=\KK[x_1,\dots,x_n]$ whenever $j\le n$.
By abuse of notation, we also call the map $\sigma$ sending the lattice point $(i,j)$ in the shifted quadrant to $x_ix_j$ the specialization map; it obviously bears a similar flavor as that in \Cref{Spec}. A monomial ideal $I$ of degree $2$ is called \emph{shifted stable} if
\[
I=I_{\lambda-\mu}:=(\sigma(t)\mid t\in D_{\lambda-\mu})
\]
for some connected quasi-Ferrers diagram $D_{\lambda-\mu}$. It is easy to verify that stable ideals are shifted stable with respect to some connected quasi-Ferrers diagram $D_{\lambda-\mu}$, where $\mu$ takes the form $(0,1,\dots,h-1)$. 

\begin{Example}
    \label{quasi-diag-exam}
    In \Cref{Quasi-Diagram}, we have a connected shifted quasi-Ferrers diagram $D_{\lambda-\mu}$ with $\lambda=(6,5,6,7)$ and $\mu=(1,3,2,3)$. The corresponding shifted stable ideal is
    \[
    I=(x_1x_2,x_1x_3,x_1x_4,x_1x_5,x_1x_6,x_2x_4,x_2x_5,x_3x_4,x_3^2,x_3x_5,x_3x_6,x_4^2,x_4x_5,x_4x_6,x_4x_7).
    \]
    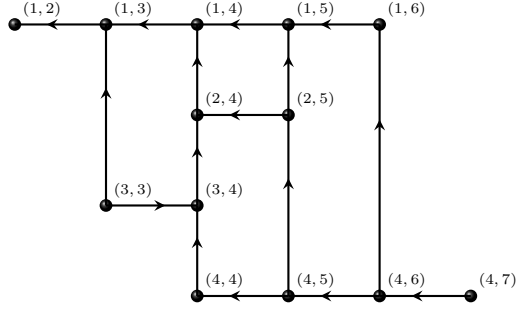
\begin{figure}[h]
        \begin{center}
            \begin{tikzpicture} [thick, scale=1.2, every node/.style={scale=0.8}]]
                \shade [shading=ball, ball color=black]  (2,-1) circle (.07) node [above right] {\scriptsize$(1,2)$};
                \shade [shading=ball, ball color=black]  (3,-1) circle (.07) node [above right] {\scriptsize$(1,3)$};
                \shade [shading=ball, ball color=black]  (3,-3) circle (.07) node [above right] {\scriptsize$(3,3)$};
                \shade [shading=ball, ball color=black]  (4,-1) circle (.07) node [above right] {\scriptsize$(1,4)$};
                \shade [shading=ball, ball color=black]  (4,-2) circle (.07) node [above right] {\scriptsize$(2,4)$};
                \shade [shading=ball, ball color=black]  (4,-3) circle (.07) node [above right] {\scriptsize$(3,4)$};
                \shade [shading=ball, ball color=black]  (4,-4) circle (.07) node [above right] {\scriptsize$(4,4)$};
                \shade [shading=ball, ball color=black]  (5,-1) circle (.07) node [above right] {\scriptsize$(1,5)$};
                \shade [shading=ball, ball color=black]  (5,-2) circle (.07) node [above right] {\scriptsize$(2,5)$};
                \shade [shading=ball, ball color=black]  (5,-3) circle (.07) node [above right] {\scriptsize$(3,5)$};
                \shade [shading=ball, ball color=black]  (5,-4) circle (.07) node [above right] {\scriptsize$(4,5)$};
                \shade [shading=ball, ball color=black]  (6,-1) circle (.07) node [above right] {\scriptsize$(1,6)$};
                \shade [shading=ball, ball color=black]  (6,-3) circle (.07) node [above right] {\scriptsize$(3,6)$};
                \shade [shading=ball, ball color=black]  (6,-4) circle (.07) node [above right] {\scriptsize$(4,6)$};
                \shade [shading=ball, ball color=black]  (7,-4) circle (.07) node [above right] {\scriptsize$(4,7)$};

                \draw[thick, directed] (7,-4) -- (6,-4);
                \draw[thick, directed] (6,-4) -- (5,-4);
                \draw[thick, directed] (5,-4) -- (4,-4);
                \draw[thick, directed] (4,-4) -- (4,-3);
                \draw[thick, directed] (4,-3) -- (4,-2);
                \draw[thick, directed] (4,-2) -- (4,-1);
                \draw[thick, directed] (3,-3) -- (4,-3);
                \draw[thick, directed] (3,-3) -- (3,-1);
                \draw[thick, directed] (4,-1) -- (3,-1);
                \draw[thick, directed] (3,-1) -- (2,-1);
                \draw[thick, directed] (5,-4) -- (5,-3);
                \draw[thick, directed] (5,-3) -- (5,-2);
                \draw[thick, directed] (5,-3) -- (4,-3);
                \draw[thick, directed] (5,-2) -- (4,-2);
                \draw[thick, directed] (5,-2) -- (5,-1);
                \draw[thick, directed] (6,-1) -- (5,-1);
                \draw[thick, directed] (5,-1) -- (4,-1);
                \draw[thick, directed] (6,-3) -- (6,-1);
                \draw[thick, directed] (6,-3) -- (5,-3);
                \draw[thick, directed] (6,-4) -- (6,-3);
            \end{tikzpicture}
            \caption{$\lambda-\mu=(6,5,6,7;1,3,2,3)$ with minimal quasi-Borel moves} \label{Quasi-Diagram}
        \end{center}
    \end{figure}
\end{Example}

\begin{Lemma}
    \label{shifted-stable-induction}
    Let $D_{\lambda-\mu}$ be a connected shifted quasi-Ferrers diagram containing more than one lattice point. Then, we can remove one such point from $D_{\lambda-\mu}$ and still get a connected shifted quasi-Ferrers diagram $D'$.    
\end{Lemma}

\begin{proof}
    We may assume that $\lambda-\mu=(\lambda_1,\lambda_2,\dots,\lambda_h;\;\mu_1,\mu_2,\dots,\mu_h)$.
    If $h=1$, we will remove the rightmost point $c=(1,\lambda_1)$ and the statement is clear. 
    In the following, we will assume that $h\ge 2$.
    As $D_{\lambda-\mu}$ is connected, we will consider the index
    \[
    t:=\min \left( \Set{j\in \ZZ_+\mid \mu_{h-1}<j\le \lambda_{h-1}} \cap
    \Set{j\in \ZZ_+\mid \mu_{h}<j\le \lambda_{h}} \right).
    \]
    \begin{enumerate}[a]
        \item If $\mu_h=t-1$ and $\lambda_h=t$, we will remove the only lattice point $c=(h,t)$ on the $h$-th row. In this case, we will write $\lambda'=(\lambda_1,\dots,\lambda_{h-1})$ and $\mu'=(\mu_1,\dots,\mu_{h-1})$.
        \item \label{induction-right} If $\lambda_h>t$, we will remove the rightmost lattice point $c=(h,\lambda_h)$ on the $h$-th row. In this case, we will write $\lambda'=(\lambda_1,\dots,\lambda_{h-1},\lambda_h-1)$ and $\mu'=\mu$.
        \item \label{induction-left} If $\mu_h+1<t=\lambda_h$, we will remove the leftmost lattice point $c=(h,\mu_h+1)$ on the $h$-th row. In this case, we will write $\lambda'=\lambda$ and $\mu'=(\mu_1,\dots,\mu_{h-1},\mu_h+1)$.
    \end{enumerate}
    Now, the remaining diagram will be $D_{\lambda'-\mu'}$. It is easy to see that $D_{\lambda'-\mu'}$ is still a connected shifted quasi-Ferrers diagram.
\end{proof}

\begin{Remark}
    \label{rmk:linear-order}
	A recursive application of \Cref{shifted-stable-induction} gives a linear ordering $\prec$ of the lattice points in $D_{\lambda-\mu}$. This can be made more explicit. 
    \begin{enumerate}[a]
        \item If $(i_1,j_1)$ and $(i_2,j_2)$ belong to $D_{\lambda-\mu}$ with
            $i_1<i_2$, we require $(i_1,j_1)\prec (i_2,j_2)$.
        \item Define     
            \[
                t_k\coloneqq \min \left( \Set{j\in \ZZ_+\mid \mu_{k-1}<j\le \lambda_{k-1}} \cap
                    \Set{j\in \ZZ_+\mid \mu_{k}<j\le \lambda_{k}} \right)
            \]
            for $2\le k \le h$ and make $t_1=1$.  Then, the lattice points of the same horizontal level are organized as
            \[
                (k,t_k)\prec \underbrace{(k,t_k-1)\prec \cdots \prec (k,\mu_k+1)}_{\text{this part exists only when $\mu_k+1<t_k$}}\prec \underbrace{(k,t_k+1) \prec \cdots \prec(k,\lambda_k)}_{\text{this part exists only when $\lambda_k>t_k$}}.
            \]
    \end{enumerate}
    For instance, the lattice points in \Cref{Quasi-Diagram} are organized as
    \begin{align*}
        (1,2)& \prec (1,3) \prec (1,4) \prec (1,5) \prec (1,6) \prec (2,4) \prec (2,5) \\
        &\prec (3,4) \prec (3,3) \prec (3,5) \prec (3,6) \prec (4,4) \prec (4,5) \prec (4,6) \prec (4,7).
    \end{align*}
\end{Remark}

Taking the specialization, we immediately get the following corollary.

\begin{Corollary}
    \label{Linear-Order}
    Let $I$ be a shifted stable ideal. Then, there is a linear order on the minimal monomial generators of $I$, $f_1\prec \cdots \prec f_\nu$, such that the sub-ideal $(f_1,\dots,f_k)$ is still shifted stable for each $k\in[\nu]$.
\end{Corollary}

Notice that linear order in the above result is far from unique. However, for our argument below, we will always assume that it is derived from \Cref{rmk:linear-order}.  The order of the minimal monomial generators in \Cref{quasi-diag-exam} satisfies this requirement. Notice that it is in general not the co-lexicographic order considered in Section 4.

\begin{Proposition}
    \label{shifted-stable-linear-quot}
    If $I$ is an $\bda$-determined shifted stable ideal of the polynomial ring $R$, then the generalized Newton complement dual $\widehat{I}^{[\bda]}$ has linear quotients.
\end{Proposition}

\begin{proof}
    We may assume that $I$ is the specialization of the connected shifted quasi-Ferrers diagram $D_{\lambda-\mu}$ with $\lambda-\mu=(\lambda_1,\lambda_2,\dots,\lambda_h;\;\mu_1,\mu_2,\dots,\mu_h)$. When $h=1$, the ideal $I=(x_1x_{\mu_1+1},x_1x_{\mu_1+2},\dots,x_1x_{\lambda_1})$. One can easily verify that the $\bda$-dual has linear quotients.

    In the following, we will assume that $h\ge 2$. Suppose that $G(I)=\Set{f_1\prec \cdots \prec f_\nu}$, where $f_{\nu}$ is the specialization of the lattice point $c=(h,h')$ removed in the proof of \Cref{shifted-stable-induction}. If we write $I'=(f_1,\dots,f_{\nu-1})$, by induction, it suffices to show that the colon ideal $J:=\widehat{I'}^{[\bda]}:\widehat{f}_\nu$ is linear. 

    Similar to the proof for \Cref{LinQuo}, we will consider the set
    \[
    X_\nu:=\Set{x_j\mid \text{ there exists $i\ne j$ such that $f_\nu x_i/x_j\in G(I)$}}.
    \]
    In the three cases of the proof for \Cref{shifted-stable-induction}, we have $X_\nu=\Set{x_h}$, $x_{h'=\lambda_h}\in X_\nu$ and $x_{h'=\mu_h+1}\in X_\nu$ respectively. In particular, $X_\nu\ne \varnothing$. One also easily gets $(X_\nu)\subseteq J$.

    For the converse, take a minimal monomial generator $y$ of $J$.  
    By definition, $y\widehat{f_{\nu}}=w\widehat{f_k}$ for some integer $k<\nu$ and some monomial $w\in R$. This is equivalent to saying that $yf_k=wf_{\nu}$. By the minimality of $y$, one has $\gcd(y,w)=1$. Thus, $y$ divides $f_{\nu}$.  As $f_\nu$ has degree $2$, if $y$ is not linear, this reduces to $y=f_{\nu}=x_{h}x_{h'}$. But we already have $x_{h}$ or $x_{h'}\in J$ by the previous argument. This will contradict the minimality of $y$.  Therefore, $y$ is linear. In turn, we have shown that $J\subseteq (X_{\nu})$.
\end{proof}

In the following, we will focus on a sub-class of shifted stable ideals of degree $2$. Our aim is to construct explicitly planar cellular minimal free resolutions for the dual of such ideals.

With respect to the linear order $f_1\prec \cdots \prec f_\nu$ of the minimal generators of the shifted stable ideal $I=I_{\lambda-\mu}$, given in \Cref{Linear-Order}, let $c_k=(i_k,j_k)$ be the corresponding lattice point for $f_k$ for each $k\in[\nu]$. A minimal quasi-Borel move $c_k\to c_{k'}$ is called \emph{good} if $k'<k$.  The arrows in \Cref{Quasi-Diagram} indicate such moves. For simplicity, we will call good minimal quasi-Borel moves as good moves.

\begin{observation}
    \label{obs2}
    \begin{enumerate}[a]
        \item For each pair of minimal quasi-Borel moves $c\to c'$ and $c'\to c$, exactly one of them is a good move.
        \item Each vertical minimal quasi-Borel move is good if and only if it points northward.
        \item \label{obs2-c}
            There are at most two good moves starting from $c_k$ for $k\ge 2$. When it has $2$ such moves, one of them is a horizontal good move of geometric length $1$ while the other is a vertical good move.
        \item The following are equivalent:
            \begin{enumerate}[i]
                \item all horizontal good moves point westward;
                \item $\mu$ is non-decreasing: $\mu_1 \le \mu_2 \le \cdots \le \mu_h$;
                \item in the linear order $f_1\prec \cdots \prec f_\nu$, if $f_k$ corresponds to $c_k=(i,j)$ and $f_{k'}$ corresponds to $c_{k'}=(i',j')$, then 
                    \[
                    k<k' \Leftrightarrow i< i' \text{ or $i=i'$ and $j<j'$}.
                    \]
            \end{enumerate}
            In particular, stable ideals of degree $2$ satisfy these equivalent conditions.
        \item \label{obs2-e} Suppose that all horizontal moves point westward. If $(i,j-1),(i,j),(i',j)\in D$ with $i'<i$, then $\mu_{i'}\le \mu_{i}<j-1<j\le \lambda_{i'}$. Hence $(i',j-1)\in D$.
    \end{enumerate}
\end{observation}

In the following, we will only consider those connected shifted quasi-Ferrers diagrams $D_{\lambda-\mu}$ with westward horizontal good moves.  With $I=I_{\lambda-\mu}$, for each $k\in[\nu]$, let
\[
X_k=\Set{x_j\mid \text{ there exists $i\ne j$ such that $f_{k'}=f_kx_i/x_j \in G(I)$ with $f_{k'}\prec f_k$}}.
\]
The definition for $X_\nu$ agrees with that defined in the proof for \Cref{shifted-stable-linear-quot}. If $f_k$ corresponds to the lattice point $c_k$ and there are $2$ good moves in the diagram that start from $c_k$, then obviously $|X_k|=2$. However, the converse is not true. For instance, consider the lattice point $(2,4)$ in \Cref{Quasi-Diagram} after omitting the point $(3,3)$ (so that all remaining horizontal good moves point westward). There is only one good move starting from $(2,4)$. However, the corresponding set $X$ has cardinality $2$, because of the existence of $(1,2)$ in the diagram.

\begin{Definition}
    \label{compat}
    A connected shifted quasi-Ferrers diagram $D_{\lambda-\mu}$ as discussed above is called \emph{compatible} if the following two conditions are satisfied:
    \begin{enumerate}[a]
        \item all horizontal good moves point westward;
        \item \label{compat-b} for each $k\in [\nu]$, there are $2$ good moves in the diagram that start from $c_k$ if and only if $|X_k|=2$.
    \end{enumerate}
\end{Definition}

We want to point out that the ``only if'' part of \ref{compat-b} in \Cref{compat} is trivial, as observed earlier. 

\begin{Lemma}
    Let $D=D_{\lambda-\mu}$ be a connected shifted quasi-Ferrers diagram with all horizontal moves point westward. Then, $D_{\lambda-\mu}$ is compatible if and only if for every $i'<i<j$ in $[\nu]$, 
    \begin{equation}
        \label{compatible}
        (i',i),(i,j)\in D \Longrightarrow (i,j-1)\in D.
    \end{equation}
\end{Lemma}

\begin{proof}
    First, suppose that $D$ is compatible and $c_{k'}:=(i',i),c_k:=(i,j)\in D$ as above. Then $k\ge 2$ and $|X_k|=1$ or $2$. 
    Notice that all horizontal good moves has geometric length $1$. Thus, 
    by \Cref{obs2} \ref{obs2-c}, it suffices to consider the special case when there is only one  good move starting from $c_k$ which is vertical. This implies that $x_{i}\in X_k$. However, $x_j\in X_k$ because of the existence of $c_{k'}\in D$. This forces $X_{k}=\Set{x_i,x_j}$, having cardinality $2$. By the compatibility assumption, there are $2$ good moves starting from $c_k$, a contradiction for the existence of this special case.

    Conversely, with the notations as above,  it is enough to show that if $X_k=\Set{x_i,x_j}$, then there are two good moves starting from $c_k$. 
    \begin{enumerate}[a]
        \item As $x_i\in X_k$, one has $a<i$ with $(a,j)\in D$. Let $a$ be the largest one satisfying this property. Hence $c_k\to (a,j)$ is a good move.
        \item As $x_j\in X_k$, one has $i'<j$ with $x_{i'}x_i\in G(I)$. If $i'\ge i$, then $(i,i')\in D$. As $D$ is a shifted quasi-Ferrers diagram, each row is connected. This implies that $(i,j-1)\in D$. If $i'<i$, by our assumption in \Cref{compatible}, we still have $(i,j-1)\in D$.
    \end{enumerate}
    In short, there are two good moves starting from $c_k$.  As $c_k$ is arbitrary, this shows that $D$ is compatible. 
\end{proof}

Since \Cref{compatible} is obviously satisfied by the diagrams for stable ideals, we immediately have the following corollary.

\begin{Corollary}
    If $I=I_{\lambda-\mu}$ is stable, then the diagram $D_{\lambda-\mu}$ is compatible.
\end{Corollary}

However, the converse is not true.

\begin{Example}
    \label{ComNotSat}
    Consider the shifted partition $\lambda-\mu=(3,3;1,1)$. The shifted quasi-Ferrers diagram illustrated in \Cref{Compatible-but-not-stable} is compatible and the associated shifted stable ideal is $I_{\lambda-\mu}=(x_1x_2,x_1x_3,x_2^2,x_2x_3)$. However, this ideal is not stable, even after permutations of variables. 
    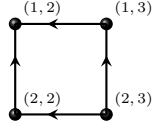
\begin{figure}[h]
        \begin{center}
            \begin{tikzpicture} [thick, scale=1.2, every node/.style={scale=0.8}]]
                \shade [shading=ball, ball color=black]  (2,-1) circle (.07) node [above right] {\scriptsize$(1,2)$};
                \shade [shading=ball, ball color=black]  (3,-1) circle (.07) node [above right] {\scriptsize$(1,3)$};    
                \shade [shading=ball, ball color=black]  (2,-2) circle (.07) node [above right] {\scriptsize$(2,2)$};
                \shade [shading=ball, ball color=black]  (3,-2) circle (.07) node [above right] {\scriptsize$(2,3)$};

                \draw[thick, directed] (3,-1) -- (2,-1);
                \draw[thick, directed] (3,-2) -- (2,-2);
                \draw[thick, directed] (3,-2) -- (3,-1);
                \draw[thick, directed] (2,-2) -- (2,-1);
            \end{tikzpicture}
            \caption{Compatible but not stable} \label{Compatible-but-not-stable}
        \end{center}
    \end{figure}
\end{Example}

From now on, we will only consider compatible diagrams. Suppose that $D=D_{\lambda-\mu}$ satisfies this requirement and $I=I_{\lambda-\mu}$ is the specialization ideal.
The good moves in $D$ induce a polyhedral cell complex $X_I$ of dimension at most $2$ as follows.
\begin{enumerate}[a]
    \item The $0$-cells are the lattice points $c_k$'s corresponding to the monomials in $G(I)$.
    \item The $1$-cells correspond to the good moves. To define the orientation, horizontal $1$-cells point westward while vertical $1$-cells point northward.
    \item The $2$-cells come from the bounded rectangular components, cut in the plane by the good moves. Each $2$-cell is uniquely determined by the lattice point on its lower right corner. This corner point $c$ comes from \Cref{obs2} \ref{obs2-c} and provides exactly $2$ good moves that start from it. It is the last corner with respect to the given linear order. The width of this cell is always $1$. It follows from \Cref{obs2} \ref{obs2-e} that the southeast and northeast corners of this cell are lattice points from $G(I)$. We endow this cell the orientation, so that the induced boundary orientation is counterclockwise.
\end{enumerate}

Suppose that $I$ is $\bda$-determined.
We will label the cells of $X_I$ as follows.
\begin{enumerate}[a]
    \item We will label the $0$-cell $(i,j)$ of $X_I$ that corresponds to $f_{i,j}=x_ix_j\in G(I)$, by the dual monomial $\widehat{f_{i,j}}=\bdx^\bda/(x_ix_j)$.
    \item For a horizontal edge corresponding to the good move $(i,j)\to (i,j-1)$, we label it by $\lcm(\widehat{f_{i,j}},\widehat{f_{i,j-1}})=\bdx^\bda/x_i$. For a vertical edge corresponding to the good move $(i,j)\to (i',j)$, we label it by $\lcm(\widehat{f_{i,j}},\widehat{f_{i',j}})=\bdx^\bda/x_j$.
    \item All $2$-cells will inevitably be labeled by $\bdx^\bda$.
\end{enumerate}

\begin{Example}
    Let $I=I_{\lambda-\mu}=(x_1x_2,x_1x_3,x_2^2,x_2x_3)$ as in \Cref{ComNotSat}, then $I_{\lambda-\mu}$ is $\bda$-determined when $\bda=(3,4,2)$. We have $\widehat{I}^{[\bda]}=(x_1^2x_2^3x_3^2,x_1^2x_2^4x_3,x_1^3x_2^2x_3^2,x_1^3x_2^3x_3)$. The labels of the cells of $X_I$ are illustrated in \Cref{Lab-Compatible-but-not-stable}.
    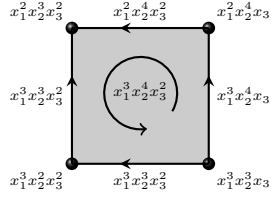
\begin{figure}[h]
        \begin{center}
            \begin{tikzpicture} [thick, scale=1.2, every node/.style={scale=0.8}]]
                \shade [shading=ball, ball color=black]  (2,-1) circle (.07) node [above left] {\scriptsize$x_1^2x_2^3x_3^2$};
                \shade [shading=ball, ball color=black]  (3.5,-1) circle (.07) node [above right] {\scriptsize$x_1^2x_2^4x_3$};    
                \shade [shading=ball, ball color=black]  (2,-2.5) circle (.07) node [below left] {\scriptsize$x_1^3x_2^2x_3^2$};
                \shade [shading=ball, ball color=black]  (3.5,-2.5) circle (.07) node [below right] {\scriptsize$x_1^3x_2^3x_3$};

                \draw[thick, directed] (3.5,-1) -- (2,-1) node [above, pos=0.5] {\scriptsize$x_1^2x_2^4x_3^2$};
                \draw[thick, directed] (3.5,-2.5) -- (2,-2.5) node [below, pos=0.5] {\scriptsize$x_1^3x_2^3x_3^2$};
                \draw[thick, directed] (2,-2.5) -- (2,-1) node [left, pos=0.5] {\scriptsize$x_1^3x_2^3x_3^2$};
                \draw[thick, directed] (3.5,-2.5) -- (3.5,-1) node [right, pos=0.5] {\scriptsize$x_1^3x_2^4x_3$};
                \draw (2.8,-1.4)+(-60:.6) [-{stealth[flex]}] arc(-30:280:.40); 
                \path[fill=black,fill opacity=0.2]    (2,-1)--(3.5,-1)--(3.5,-2.5)--(2,-2.5)--cycle;
                \path (2,-1.7)--(3.5,-1.7) node [pos=0.5] {\scriptsize$x_1^3x_2^4x_3^2$};

            \end{tikzpicture}
            \caption{$X_I$ with $I=(x_1x_2,x_1x_3,x_2^2,x_2x_3)$} \label{Lab-Compatible-but-not-stable}
        \end{center}
    \end{figure}
\end{Example}

With the orientations and the labels described above, we will have an induced cellular free complex $\calF_{X_I}$. Our aim is to show that $\calF_{X_I}$ provides a minimal free resolution for $R/\widehat{I}^{[\bda]}$.

\begin{Theorem}
    \label{compatible-stable}
    Let $I=I_{\lambda-\mu}$ be an $\bda$-determined shifted stable ideal in $R=\KK[x_1,\dots,x_n]$, generated in degree $2$. If the diagram $D=D_{\lambda-\mu}$ is compatible, then the complex $\calF_{X_I}$ provides a minimal free resolution for $R/\widehat{I}^{[\bda]}$. 
\end{Theorem}

\begin{proof}
    As in \Cref{Linear-Order}, we assume that $G(I)=\Set{f_1\prec f_2 \prec \cdots \prec f_{\nu}}$, and let $I'=(f_1,\dots,f_{\nu-1})$. Suppose that $f_{\nu}=x_{t_1}x_{t_2}$ with $t_1\le t_2$. 
    \begin{enumerate}[a]
        \item Assume that there are two good moves starting from $(t_1,t_2)$. This can only happen in the case \ref{induction-right} of the proof for \Cref{shifted-stable-induction}, as noticed by \Cref{obs2} \ref{obs2-c}.  Whence, $t_1< t_2$ and we will remove the rightmost lattice point from the last row. The two good moves are: 
            \[
            e_1:(t_1,t_2)\to (t_1,t_2-1) \quad\text{ and }\quad e_2:(t_1,t_2)\to (t_0,t_2), 
            \]
            for $t_0<t_1$. Then $f_{\nu}$ contributes to $X_{I}$ the $0$-cell $(t_1,t_2)$, the two edges corresponding to $e_1$ and $e_2$, and the rectangular $2$-cell $Z$ with $e_1$ and $e_2$ as its two adjacent sides. We may assume that the other two sides of $Z$ are $e_3: (t_0,t_2) \to (t_0,t_2-1)$ and 
            \[
            p_0:=(t_0'=t_1,t_2-1)\stackrel{e_1'}{\to} p_1:=(t_1',t_2-1) \stackrel{e_2'}{\to} \cdots \stackrel{e_s'}{\to} p_s:=(t_s'=t_0,t_2-1).
            \]
            In the above quasi-Borel walk, we assume that each quasi-Borel move minimal. As a result, they are good moves. 
            Of course, $e_3$ is also a good move. 
            We will write $q_1$ for the point $(t_1,t_2)$ and $q_2$ for the point $(t_0,t_2)$. Now, the overall picture is illustrated in \Cref{Rectangle-Stable}.

            \begin{figure}[h]
                \begin{center}
                    \begin{tikzpicture}[thick, scale=1, every node/.style={scale=0.8}]
                        \shade [shading=ball, ball color=black]  (5,-1) circle (.07) node [left] {\scriptsize$p_s(t_0=t_s',t_2-1)$};
                        \shade [shading=ball, ball color=black]  (5,-3) circle (.07) node [left] {\scriptsize$p_2(t_2',t_2-1)$};
                        \shade [shading=ball, ball color=black]  (4,-2) node {$\vdots$};
                        \shade [shading=ball, ball color=black]  (5,-4) circle (.07) node [left] {\scriptsize$p_1(t_1',t_2-1)$};
                        \shade [shading=ball, ball color=black]  (5,-5) circle (.07) node [left] {\scriptsize$p_0(t_0'=t_1,t_2-1)$};
                        \shade [shading=ball, ball color=black]  (6,-1) circle (.07) node [right] {\scriptsize$q_2(t_0,t_2)$};
                        \shade [shading=ball, ball color=black]  (6,-5) circle (.07) node [right] {\scriptsize$q_1(t_1,t_2)$};

                        \draw[thick, directed] (6,-5) -- (5,-5) node[midway,below] {\scriptsize$e_1$};
                        \draw[thick, directed] (5,-5) -- (5,-4) node[midway,left] {\scriptsize$e_1'$};
                        \draw[thick, directed] (5,-4) -- (5,-3) node[midway,left] {\scriptsize$e_2'$};
                        \draw[thick, directed] (5,-3) -- (5,-1);
                        \draw[thick, directed] (6,-1) -- (5,-1) node[midway,above] {\scriptsize$e_3$};
                        \draw[thick, directed] (6,-5) -- (6,-1) node[midway,right] {\scriptsize$e_2$};

                        \path[fill=black,fill opacity=0.2]    (5,-1)--(6,-1)--(6,-5)--(5,-5);
                        \draw (5.5,-3)+(-30:.2) [-{stealth[flex]}] arc(-30:280:.2); 

                    \end{tikzpicture}
                    \caption{The new $2$-cell $Z$} \label{Rectangle-Stable}
                \end{center}
            \end{figure}
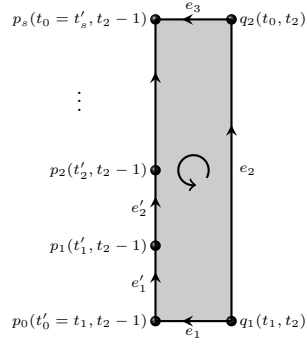

            Notice that from the proof of \Cref{shifted-stable-linear-quot}, we know $\widehat{I'}^{[\bda]}:\widehat{I}^{[\bda]}=(x_{t_1},x_{t_2})$. This colon ideal can be resolved minimally by the Koszul complex:
            \[
            0\to R(Z,x_{t_1}x_{t_2}) \xrightarrow{
            \begin{pmatrix}
                -x_{t_2} \\ x_{t_1}
            \end{pmatrix}
            } R(e_2,x_{t_1})\oplus R(e_1,x_{t_2})\xrightarrow{
            \begin{pmatrix}
                x_{t_1} & x_{t_2}
            \end{pmatrix}
            } R(\varnothing,1).
            \]
            Meanwhile, by induction, we may assume that $\calF_{X_{I'}}$ gives a minimal cellular free resolution for $\widehat{I'}^{[\bda]}$:
            \[
            0 \to (\calF_{X_{I'}})_2 \xrightarrow{\partial_2} (\calF_{X_{I'}})_1  \xrightarrow{\partial_1} (\calF_{X_{I'}})_0 \xrightarrow{\partial_0} R.
            \]
            We have the comparison map for the mapping cone as follows, induced naturally from the polyhedral cell complex $X_I$:
            \[
            \begin{tikzcd}[ampersand replacement=\&]
                0 \arrow[r] \&  0\arrow[r] \arrow[d,"0"] \& 
                R(Z,\bdx^{\bda})  \arrow[r, "{\begin{pmatrix} 
                    -x_{t_2} \\x_{t_1} 
                \end{pmatrix}}"] \arrow[d,"\varphi_2"]
                \&
                R(e_2,\widehat{f}_{\nu}x_{t_1})\oplus R(e_1,\widehat{f}_{\nu}x_{t_2}) \arrow[r,"{(x_{t_1}\,  x_{t_2})}"] \arrow[d,"\varphi_1"]
                \& R(\varnothing,\widehat{f}_{\nu})  \arrow[d,"\widehat{f}_\nu"]  \\ 
                0 \arrow[r] \& (\calF_{X_{I'}})_2 \arrow[r,"\partial_2"]  \& (\calF_{X_{I'}})_1 \arrow[r,"\partial_1"] \& (\calF_{X_{I'}})_0  \arrow[r,"\partial_0"] \& R.
            \end{tikzcd}
            \]
            Here, the minimal free resolution in the top is shifted in multidegree by $\widehat{f}_{\nu}$ so that the comparison map is of degree $0$. For our homomorphisms $\varphi_1$ and $\varphi_2$, the essential direct summands of $(\calF_{X_{I'}})_0$ and $(\calF_{X_{I'}})_1$ are
            \[
            R(q_2,\bdx^{\bda}/(x_{t_2}x_{t_0}))\oplus R(p_0,\bdx^{\bda}/(x_{t_2-1}x_{t_0'}))\oplus \cdots \oplus R(p_s,\bdx^{\bda}/(x_{t_2-1}x_{t_s'}))
            \]
            and
            \[
            R(e_3,\bdx^{\bda}/x_{t_0})\oplus R(e_1',\bdx^{\bda}/x_{t_2-1})\oplus \cdots \oplus R(e_s',\bdx^{\bda}/x_{t_2-1})
            \]
            respectively.
            Written in matrix form, the submatrix for $\varphi_2$ with respect to these submodules is the column vector 
            \[
            (x_{t_0},-x_{t_2-1},-x_{t_2-1},\dots,-x_{t_2-1})^\upT.
            \]
            The submatrix for $\varphi_1$ is 
            \[
            \begin{pmatrix}
                x_{t_0} & 0 & 0 & \cdots \\
                0 & x_{t_2-1} & 0 & \cdots 
            \end{pmatrix}^\upT.
            \]
            The submatrix for $\partial_0$ is the row vector 
            \[
            \begin{pmatrix}
                \bdx^{\bda}/(x_{t_2}x_{t_0}) &  \bdx^{\bda}/(x_{t_2-1}x_{t_0'}) & \cdots & \bdx^{\bda}/(x_{t_2-1}x_{t_s'})
            \end{pmatrix}.
            \]
            The submatrix for $\partial_1$ is 
            \[
            \begin{pmatrix}
                -x_{t_2} \\
                0 & -x_{t_1}=-x_{t_0'} \\
                & x_{t_1'} & -x_{t_1'} \\
                & & x_{t_2'} & -x_{t_2'} \\
                & & & \ddots & \ddots \\
                & & & & x_{t_{s-1}'} & -x_{t_{s-1}'} \\
                x_{t_2-1} & & & & & x_{t_s'}=x_{t_0}
            \end{pmatrix}.
            \]
            By an easy checking of the commutativity of the diagrams using matrices, we know the claimed comparison map is valid. Therefore, the mapping cone of this comparison map, as defined in \Cref{mapping-cone}, provides a free resolution of $R/\widehat{I}^{[\bda]}$:
            \[
            0 \to 
            R(Z,\bdx^{\bda})\oplus (\calF_{X_{I'}})_2  \to  
            R(e_2,\widehat{f}_{\nu}x_{t_1})\oplus R(e_1,\widehat{f}_{\nu}x_{t_2})\oplus (\calF_{X_{I'}})_1 \to R(\varnothing,\widehat{f}_{\nu})\oplus (\calF_{X_{I'}})_0 \to R.
            \]
            The resolution is minimal by checking the maps described above. One can also check with ease that this resolution agrees with the cellular free resolution $\calF_{I_X}$  with the desired differentials.
        \item Assume that there is only one good move staring from the lattice point $(t_1,t_2)$. The proof is similar and simpler. \qedhere
    \end{enumerate}
\end{proof}

In the following, we will focus on classical stable ideals. Notice that if $I=I_{\lambda-\mu}$ is stable, the diagram $D=D_{\lambda-\mu}$ is compatible.

Let $T$ be a non-empty set of monomials in $R=\KK[x_1,\dots,x_n]$. We write $\Sta(T)$ for the unique smallest stable ideal in $R$ that contains $T$.  If $J$ is a stable ideal in $R$, we write $\StaG(J)$ for the unique smallest set $T$ of monomials in $R$ with $\Sta(T)=J$.  
\Cref{Sta741} illustrates a stable ideal of degree $2$ in $R=\KK[x_1,\dots,x_7]$ with good moves.

\begin{figure}[h]
    \begin{center}
        \begin{tikzpicture} [thick, scale=1.2, every node/.style={scale=0.8}]]
            \shade [shading=ball, ball color=black]  (1,-1) circle (.07) node [above right] {\scriptsize$(1,1)$};
            \shade [shading=ball, ball color=black]  (2,-1) circle (.07) node [above right] {\scriptsize$(1,2)$};
            \shade [shading=ball, ball color=black]  (2,-2) circle (.07) node [above right] {\scriptsize$(2,2)$};
            \shade [shading=ball, ball color=black]  (3,-1) circle (.07) node [above right] {\scriptsize$(1,3)$};
            \shade [shading=ball, ball color=black]  (3,-2) circle (.07) node [above right] {\scriptsize$(2,3)$};
            \shade [shading=ball, ball color=black]  (3,-3) circle (.07) node [above right] {\scriptsize$(3,3)$};
            \shade [shading=ball, ball color=black]  (4,-1) circle (.07) node [above right] {\scriptsize$(1,4)$};
            \shade [shading=ball, ball color=black]  (4,-2) circle (.07) node [above right] {\scriptsize$(2,4)$};
            \shade [shading=ball, ball color=black]  (4,-3) circle (.07) node [above right] {\scriptsize$(3,4)$};
            \shade [shading=ball, ball color=black]  (4,-4) circle (.07) node [above right] {\scriptsize$(4,4)$};
            \shade [shading=ball, ball color=black]  (5,-1) circle (.07) node [above right] {\scriptsize$(1,5)$};
            \shade [shading=ball, ball color=black]  (5,-2) circle (.07) node [above right] {\scriptsize$(2,5)$};
            \shade [shading=ball, ball color=black]  (5,-4) circle (.07) node [above right] {\scriptsize$(4,5)$};
            \shade [shading=ball, ball color=black]  (6,-1) circle (.07) node [above right] {\scriptsize$(1,6)$};
            \shade [shading=ball, ball color=black]  (6,-4) circle (.07) node [above right] {\scriptsize$(4,6)$};
            \shade [shading=ball, ball color=black]  (7,-4) circle (.07) node [above right] {\scriptsize$(4,7)$};

            \draw[thick, directed] (7,-4) -- (6,-4);
            \draw[thick, directed] (6,-4) -- (5,-4);
            \draw[thick, directed] (5,-4) -- (4,-4);
            \draw[thick, directed] (4,-4) -- (4,-3);
            \draw[thick, directed] (4,-3) -- (4,-2);
            \draw[thick, directed] (4,-2) -- (4,-1);
            \draw[thick, directed] (4,-3) -- (3,-3);
            \draw[thick, directed] (3,-3) -- (3,-2);
            \draw[thick, directed] (4,-2) -- (3,-2);
            \draw[thick, directed] (3,-2) -- (2,-2);
            \draw[thick, directed] (2,-2) -- (2,-1);
            \draw[thick, directed] (2,-1) -- (1,-1);
            \draw[thick, directed] (3,-2) -- (3,-1);
            \draw[thick, directed] (4,-1) -- (3,-1);
            \draw[thick, directed] (3,-1) -- (2,-1);
            \draw[thick, directed] (5,-4) -- (5,-2);
            \draw[thick, directed] (5,-2) -- (4,-2);
            \draw[thick, directed] (5,-2) -- (5,-1);
            \draw[thick, directed] (6,-1) -- (5,-1);
            \draw[thick, directed] (5,-1) -- (4,-1);
            \draw[thick, directed] (6,-4) -- (6,-1);
        \end{tikzpicture}
        \caption{$\Sta(\{x_2x_5,x_1x_6,x_4x_7\})$ with good moves} \label{Sta741}
    \end{center}
\end{figure}

Now, suppose that the stable ideal $I$ coming from the diagram $D_{\lambda-\mu}$ satisfies
\begin{equation} \label{StaG-EQ}
    \StaG(I)=\Set{m_k=x_{k_{1}}x_{k_{2}}|1\le k\le g \text{ with }k_{1}\le k_{2}}.
\end{equation}
Notice that $h=\max\Set{k_{1}|1\le k\le g}$ in $\lambda-\mu=(\lambda_1,\dots,\lambda_h;\mu_1,\dots,\mu_h)$.  For each $j\in [n]$, we define
\[
Y_j:=\Set{i\in[j]| x_ix_j\in G(I)}.
\] 

\begin{observation}
    \label{obs1}
    \begin{enumerate}[a]
        \item For every $x_i^2\in G(I)$ with $2\le i\le h$, there exists exactly one good move starting from the corresponding lattice point $(i,i)$. It points northward.
        \item \label{obs1-b} For every $x_ix_j\in G(I)$ with $i<j$, there always exists a westward good move of geometric length $1$ that starts from the corresponding lattice point $(i,j)$. There exists another good move starting from $(i,j)$ precisely when the $|Y_j|\ge 2$ and $i>\min(Y_j)$. This good move points northward.
        \item \label{obs1-c} With $\StaG(I)$ as in \Cref{StaG-EQ}, we have the containment:
            \[
            Y_1 \subset Y_2 \subset \cdots \subset Y_h \supseteq Y_{h+1}\supseteq \cdots \supseteq Y_e 
            \]
            with 
            \[
            e=\max\Set{k_{2}|k\in[g]}=\max\Set{\lambda_i\mid i\in [h]}.
            \]
            To be more precise, for $1\le j\le h$, $Y_j=\Set{1,2,\dots,j}$. For $h+1\le j\le e$, $Y_j=\Set{k_{1} | 1\le k\le g \text{ with } k_{2}\ge j}$.   
    \end{enumerate}
\end{observation}

\begin{Remark}
    \label{bettiStable}
    Using the notations in \Cref{obs1} \ref{obs1-c} and \Cref{compatible-stable}, if we check backwards, then among the $\nu$ minimal monomial generators of $I$, there are
    \[
    w_2:=0+0+1+2+\cdots+(h-2)+(|Y_{h+1}|-1)+\cdots+(|Y_{e}|-1)
    \]
    of them contribute $2$-cells in the induction process in the proof for \Cref{compatible-stable}. There are 
    \[
    w_1:=0+2(h-1)+(e-h)
    \]
    of them only contribute $1$-cells. The remaining initial point of course only contributes the point itself. Combining the proof for \Cref{shifted-stable-linear-quot} and \cite[Corollary 8.2.2]{HH}, we know the Betti numbers of $\widehat{I}^{[\bda]}$ satisfies the formula:
    \[
    \beta_i(\widehat{I}^{[\bda]})=w_2\binom{2}{i}+w_1\binom{1}{i}, \quad i=1,2.
    \]
    Of course, these formulae can be computed directly with ease from the graph or by our induction proof for \Cref{shifted-stable-linear-quot}.
\end{Remark}

\begin{acknowledgement*}
    The authors thank Alberto Corso for sharing his experiments which inspires this work. The authors thank Claudia Polini and Sonja Mapes for useful discussions and careful readings of the manuscript. They also want to express heartfelt thanks to the  reviewer for the suggestions that greatly improved this paper. The third author is partially supported by the ``Fundamental Research Funds for the Central Universities''.
\end{acknowledgement*}

\end{document}